\documentclass[11pt, reqno]{amsart}
\usepackage{amsbsy}
\usepackage{amsfonts}
\usepackage{amsmath}
\usepackage{amsthm}
\usepackage{amssymb}
\usepackage{enumitem}
\usepackage{ytableau}

\usepackage{hyperref}
\usepackage{color}
\usepackage{mwe}
\usepackage{float}

\usepackage{etoolbox}

\newcommand{\RR}{{\mathbb R}}
\theoremstyle{plain}

\newtheorem{thm}{Theorem}[section]
\newtheorem{lem}[thm]{Lemma}
\newtheorem{prop}[thm]{Proposition}
\newtheorem{cor}[thm]{Corollary}

\newtheorem{dfn}[thm]{Definition}
\newtheorem{remark}[thm]{Remark}

\newcommand\supp{\mathrm{supp}\;}
\newcommand\sgn{\mathrm{sgn}}

\newcommand\distinctsum{\sum_{\substack{j_1,...,j_n \\ \mathrm{distinct}}}}
\newcommand\Ee{\mathbb{E}}

\newcommand\var{\mathrm{var}}
\definecolor {UMblue}  {RGB}{0, 39, 76}
\definecolor {UMmaize} {RGB}{255, 203, 5}

\definecolor {color_b}{RGB}{255,0,0}
\definecolor {color_c}{RGB}{20, 200, 30}
\definecolor {color_a}{RGB}{0,0,255}

\definecolor{lgreen} {RGB}{180,210,100}
\definecolor{dblue}  {RGB}{20,66,129}
\definecolor{ddblue} {RGB}{11,36,69}
\definecolor{lred}   {RGB}{220,0,0}
\definecolor{nred}   {RGB}{224,0,0}
\definecolor{norange}{RGB}{230,120,20}
\definecolor{nyellow}{RGB}{255,221,0}
\definecolor{ngreen} {RGB}{98,158,31}
\definecolor{dgreen} {RGB}{78,138,21}
\definecolor{nblue}  {RGB}{28,130,185}
\definecolor{jblue}  {RGB}{20,50,100}

\begin{document}

\title[Band-limited mimicry]{Band-limited mimicry of point processes by point processes supported on a lattice}
\author{Jeffrey C. Lagarias}
\address{Department of Mathematics, University of Michigan, Ann Arbor, MI 48109-1043}
\email{lagarias@umich.edu}

\author{Brad Rodgers}
\address{Department of Mathematics and Statistics, Queen's University, Kingston, ON, Canada  K7L 3N6}
\email{brad.rodgers@queensu.ca}

\begin{abstract}
We say that one point process on the line $\RR$ mimics another at a bandwidth $B$ if for each $n \ge 1$ the two point processes have $n$-level correlation functions that agree when integrated against  all band-limited test functions on bandwidth $[-B, B]$. This paper asks the question of for what values $a$ and $B$ can a given point process on the real line be mimicked at bandwidth $B$ by a point process supported on the lattice $a\mathbb{Z}$. For Poisson point processes we give a complete answer for allowed parameter ranges $(a,B)$, and for the sine process we give existence and nonexistence regions for parameter ranges. The results for the sine process have an application to the Alternative Hypothesis regarding the scaled spacing of zeros of the Riemann zeta function, given in a companion paper.
\end{abstract}

\maketitle

\section{Introduction}
\label{sec:intro}

\subsection{Objective}
\label{subsec:outline}

In this paper we ask the following question: how well can the statistics of a point process on the real line $\mathbb{R}$ be mimicked by the statistics of a point process restricted to a lattice $a\mathbb{Z} = \{aj:\, j\in \mathbb{Z}\}$? The statistics we consider are correlation functions, and what we mean by `mimicking' is perfect agreement of the correlation functions of the two processes when integrated against band-limited Schwartz functions of a specified bandwidth, explained further below. We give an analysis of the mimicking problem for two distinct point processes, the Poisson process and the sine process, and uncover some surprising mismatches between the two. 

This problem has its origins in a problem regarding the zeros of the Riemann zeta-function, which we discuss at the end of the introduction and treat more fully in a companion paper \cite{LaRo19}.

\subsection{Background and conventions for point processes}
\label{subsec:background}

We first recall the definition of a point process, and fix notation. A good reference for point processes with conventions similar to ours is Hough et al. \cite{HoKrPeVi09}. Other basic references for point processes include \cite{AnGuZe10, Gr77, Sos00}. 

A point process is a recipe to randomly lay down points in some topological space. In more formal terms: we consider a locally compact separable topological space $\mathfrak{X}$; in fact for us $\mathfrak{X}$ will always be $\mathbb{R}$ or $a\mathbb{Z}$ for some $a>0$, equipped with the Euclidean topology, which is the discrete topology on $a\mathbb{Z}$. A \textbf{point configuration} $u$ in $\mathfrak{X}$ is a sequence of elements ${ u} := (u_j)_{j \in \mathbb{Z}}$ with $u_i \in \mathfrak{X}$ for all $i\in \mathbb{Z}$. For $u$ a configuration, and $V \subset \mathfrak{X}$, we use the notation
$$
\#_V(u) := \#\{i:\, u_i \in V\}
$$
to denote the number of elements of the configuration $u$ inside $V$. We allow repeated values $u_j=u_k $ with $j \ne k$, and count them with multiplicity. We let the configuration space $\mathrm{Conf}(\mathfrak{X})$ be the set of locally finite configurations, that is 
$$
\mathrm{Conf}(\mathfrak{X}) := \{ u:\, \#_K(u) < +\infty \; \textrm{for all compact}\; K\}.
$$
We let $\mathfrak{M}$ be the smallest topology on $\mathrm{Conf}(\mathfrak{X})$ that contain all cylinder sets $C_m^V$, where 
$$
C_m^V := \{u \in \mathrm{Conf}(\mathfrak{X}):\, \#_V(u) = m\}, 
$$ 
where $V$ is any bounded Borel set and $m$ is any non-negative integer.

We let $\mathcal{B}(\mathfrak{M})$ be the Borel $\sigma$-algebra generated by $\mathfrak{M}$. A \textbf{point process} on $\mathfrak{X}$ is a random element $u$ taking values in $(\mathrm{Conf}(\mathfrak{X}), \mathcal{B}(\mathfrak{M}))$. With this definition, the sets
$$
\{u:\, \#_{B_1}(u) = m_1, \, \#_{B_2}(u) = m_2, \, ... , \#_{B_n}(u) = m_n\}
$$
are measurable events, for any finite collection of Borel subsets $B_1, B_2, ..., B_n$ of $\mathfrak{X}$ and for any finite collection of non-negative integers $m_1,...,m_n$. This definition allows points to coincide; they may have a finite multiplicity.  A point process is said to be {\bf simple} if (with probability one)  any configuration has $u_i \ne u_j$ if $i \ne j$. 

In this paper we  specialize to the case that the space is $\mathfrak{X}$ is $\mathbb{R}$ or $a\mathbb{Z}$ for some $a>0$.

For the point processes we will be interested in, we will require an additional condition.\medskip

\noindent{\bf Uniform Local Moments Condition.} 
{\em  
For each $n \geq 1$ there exists a constant $C_n < \infty$ such that
\begin{equation}
\label{eq:uniform_moments_def}
\Ee [ \, (\#_{[L,L+1]}(u))^n ]\leq C_n, \quad \textrm{for all}\; L\in \mathbb{R}.
\end{equation}
Here $C_n$ depends on $n$ but does not depend on $L$.\medskip

We say that a point process satisfying \eqref{eq:uniform_moments_def} has \textbf{uniform local moments}, and refer to it subsequently as a \textbf{u.l.m. point process}.
}

Given a point process on $\mathbb{R}$, for any $n\geq 1$ and any $\phi \in C_c(\mathbb{R}^n)$, the sum
\begin{equation}
\label{eq:corrsum}
\distinctsum \phi(u_{j_1},...,u_{j_n})
\end{equation}
defines a random variable (that is, a measurable mapping from $\mathrm{Conf}(\mathfrak{X})$ to $\mathbb{C}$). In the case that our point process has uniform local moments, the Riesz representation theorem implies for that measure that for all $n\geq 1$ that there exists a unique measure $\rho_n$ on $\mathbb{R}^n$ such that
\begin{equation}
\label{eq:corrmeasures}
\Ee \distinctsum \phi(u_{j_1},...,u_{j_n}) = \int_{\mathbb{R}^n} \phi(x_1,....,x_n) \, d\rho_n(x_1,...,x_n),
\end{equation}
for all $\phi \in C_c(\mathbb{R}^n)$. (See Theorem \ref{thm:moments_implies_correlation} in Appendix \ref{subsec:existence_corrmeasures}.) In the case that $\mathfrak{X} = a\mathbb{Z}$, the measure $\rho_n$ will be supported on $(a\mathbb{Z})^n$. The measure $\rho_n$ is called the \textbf{$n$-level correlation measure of the process $u$}. (The name \textbf{$n$-level joint intensity measure} is used interchangeably in some literature, e.g. \cite[Chap. 3]{HoKrPeVi09}.)

We recall  the well-known fact that if $V$ is any Borel subset of $\mathbb{R}$ (or $a\mathbb{Z}$), we have
\begin{equation}
\label{eq:factorial_counts}
\distinctsum \mathbf{1}_V(u_{j_1})\cdots \mathbf{1}_V(u_{j_n}) = \prod_{i=0}^{n-1} (\#_V(u) - i)
\end{equation}
where $\mathbf{1}_V$ is the indicator function of the set $V$. In consequence 
\begin{equation}
\label{eq:factorial_moments}
\Ee\, \prod_{i=0}^{n-1} (\#_V(u) - i) = \int_{V^n} d\rho_n(x_1,...,x_n).
\end{equation}
From this it follows that a u.l.m.  point process has finite constants $A_n$ such that $\rho_n([L,L+1]^n) \leq A_n$.

The u.l.m. condition on  a point process allows us to extend \eqref{eq:corrsum} to a slightly wider class of functions $\phi$ than $C_c(\mathbb{R}^n)$. Let $\mathcal{S}(\mathbb{R}^n)$ be the Schwartz class of functions on $\mathbb{R}^n$.

\begin{prop}
\label{prop:toSchwartz}
Let $u$ be a u.l.m. point process on $\mathbb{R}$ and let $\rho_n$ be the $n$-level correlation measure of the process $u$ (defined by \eqref{eq:corrmeasures} for all $\phi \in C_c(\mathbb{R}^n)$). Then for all $n\geq 1$ and $\eta \in \mathcal{S}(\mathbb{R}^n)$, the sum
$$
\distinctsum \eta(u_{j_1},...,u_{j_n})
$$
converges almost surely and defines an integrable random variable, with
\begin{equation}
\label{eq:corrmeasuresSchwartz}
\Ee \distinctsum \eta(u_{j_1},...,u_{j_n}) = \int_{\mathbb{R}^n} \eta(x_1,....,x_n) \, d\rho_n(x_1,...,x_n).
\end{equation}
\end{prop}

Proposition \ref{prop:toSchwartz} is proved via a simple limiting argument combined with the dominated convergence theorem. Theorem \ref{thm:correlations_to_rapidlydecaying} in Appendix \ref{subsec:compactlysupported_to_schwartz} gives a slightly more general result with a full proof.

\begin{remark}
\label{rem:12}
It is possible for two distinct point processes share the same correlation functions for all $n \ge 1$. For instance, if $X$ and $Y$ are random variables taking values in the natural numbers which have the same moments but different distributions -- see \cite[Sec. 11.7]{St14} for a construction -- let $u$ be the point process consisting of $X$ points at the origin (and no other points) and $v$ be the point process consisting of $Y$ points at the origin (and no other points). Then $u$ and $v$ will have the same correlation measures but different distributions.

This phenomenon is not the usual situation: if  a point process has uniform local moments whose constants $C_n$ in \eqref{eq:uniform_moments_def} do not grow too quickly with $n$, then any other point process that has the same correlation measures $\rho_n$ for $n \ge 1$ must be identical in distribution. See \cite[Theorem 2] {Le73} and \cite[Remark 1.2.4]{HoKrPeVi09}. One may make a comparison between this fact and the classical moment problem for random variables, see Lenard \cite[p. 242]{Le75}, and for the moment problem \cite{ShTa43},  \cite{Ak65}, \cite{Si98}.
\end{remark}

\subsection{Statement of the problem}
\label{subsec:theproblem}

Throughout the paper we use the convention that the Fourier transform of $\eta(x)$  on $\mathbb{R}^n$ is given by 
$$
\hat{\eta}(\xi) = \int_{\mathbb{R}^n}  \eta(x) e(-x\cdot \xi)\, dx,
$$ 
where $e(y) = e^{2\pi i y}$ and $x \cdot \xi = x_1 \xi_1 + \cdots + x_n\xi_n$. 

We make the following definition.

\begin{dfn}
\label{dfn:bandlimited_mimicry}
Let $u$ and $v$ be u.l.m. point processes in $\mathbb{R}$, and let $B > 0$. Suppose that for each $n\geq 1$ and all $\eta \in \mathcal{S}(\mathbb{R}^n)$ whose Fourier transform  $\hat{\eta}$ is supported in $[-B,B]^n$, we have
\begin{equation}
\label{eq:bandlimited_equality}
\Ee \distinctsum \eta(u_{j_1},...,u_{j_n}) = \Ee \distinctsum \eta(v_{j_1},...,v_{j_n}).
\end{equation}
Then we say that \textbf{$v$ mimics $u$ at the bandwidth $[-B, B]$} (resp. \textbf{$v$ mimics $u$ at the bandwidth $B$}).
\end{dfn}

The mimicry relation is an equivalence relation: it is reflexive, symmetric and transitive. The symmetry  property is if $v$ mimics $u$ at bandwidth $B$, then $u$ mimics $v$ at bandwidth $B$.

We have been motivated to consider this definition by an application to number theory described in Subsection \ref{subsec:alterhyp}.

A point process is said to be {\bf supported} on $a\mathbb{Z}$ if all configurations lie in $a\mathbb{Z}$. We ask the following question in general:

\begin{quote}
{\bf Band-limited Mimicry Problem.} {\em  For a given u.l.m. point process $u$ in $\mathbb{R}$, for what values $a$ and $B$ does there exist a  u.l.m. point process $u^\ast$ supported on the lattice $a\mathbb{Z}$ such that $u^\ast$ mimics $u$ at the bandwidth $B$?}
\end{quote}

In this problem we do not require either the point processes we consider $u$ or $u^{\ast}$ to be simple.\\

In the band-limited mimicry problem for a process $u$ we are given partial information about the $n$-level correlation measures for a putative point process $u^\ast$ supported on $a\mathbb{Z}$. A major difficulty in resolving the problem is that not all collections of measures $\rho^\ast_n$ are realizable as the correlation measures of some point process. Determining which collections of measures are in fact the correlation measures of some point process is referred to as the \emph{realizability of point processes}. Abstract criteria for the realizability of a point process were  given by Lenard \cite[Theorem 4.1]{Le75}  in terms of correlation measures. These criteria are  hard to apply in practice. Lenard also  specified a large set of inequalities that  correlation functions must satisfy, which provide a possible mechanism to prove non-realizability, e.g. \cite[Propositions 3.4--3.8]{Le75}. The realizability problem  has  more recently been the subject of considerable work \cite{KuLeSp07,KuLeSp11,CaInKu16}.

It is not  clear  for which point processes band-limited mimicry is possible at all (namely,  for some $a, B >0$). This paper exhibits some processes where band-limited mimicry is possible and establishes limits on allowable mimicry parameters $(a, B)$.

\subsection{Sampling and interpolation}
There is a certain relation between the band-limited mimicry problem and the classical problems of sampling and interpolating a signal. Below by saying that a function is band-limited on $\mathbb{R}^n$, we mean that function has compactly supported Fourier transform.

Indeed, the sampling theorem (see Grafakos \cite[Thm. 5.6.9]{Gr14}) tells us that a band-limited function $\eta \in \mathcal{S}(\mathbb{R}^n)$ which satisfies $\supp \hat{\eta} \subset [-1/2a, 1/2a]^n$ can be reconstructed (by interpolation) from its sample values on the  lattice  $a \mathbb{Z}$, by the Whittaker-Shannon interpolation formula 
\begin{equation}\label{eq:interpolation}
\eta(x) = \sum_{k \in (a\mathbb{Z})^n} \eta(k) \prod_{i=1}^n S\Big(\frac{x_i-k_i}{a}\Big),
\end{equation}
where $S(x)$ is a sinc-function, defined by
\begin{equation}
\label{eq:sinc_def}
S(x) = \begin{cases} 
\frac{\sin\, \pi x}{\pi x} & x\neq 0 \\
\,\,\,1 & x=0.
\end{cases}
\end{equation}

Therefore, given a Schwartz function $\eta \in \mathcal{S}(\mathbb{R}^n)$ having  $\supp \hat{\eta} \subset [1/2a, 1/2a]^n$, for a u.l.m. point process's $n$-point correlation measures $\rho_n$, we have 
\begin{equation}
\label{eq:sampling_correlations0}
\int_{\mathbb{R}^n} \eta(x) d\rho_n(x) =  \int_{\mathbb{R}^n} \sum_{k \in (a\mathbb{Z})^n} \eta(k) \prod_{i=1}^n S\Big(\frac{x_i-k_i}{a}\Big) \, d\rho_n(x).
\end{equation}
Under mild hypotheses on $\rho_n$ we can interchange the  sum and integral on the right side to obtain
\begin{equation}
\label{eq:sampling_correlations1}
\int_{\mathbb{R}^n} \eta(x) d\rho_n(x) = \sum_{k \in (a\mathbb{Z})^n} \eta(k) \int_{\mathbb{R}^n} \prod_{i=1}^n S\Big(\frac{x_i-k_i}{a}\Big) \, d\rho_n(x).
\end{equation}

That is, we have 
\begin{align}
\label{eq:sampling_correlations}
\int_{\mathbb{R}^n} \eta(x) d\rho_n(x)
= \int_{\mathbb{R}^n} \eta(x) d\rho_n'(x),
\end{align}
where $\rho_n'(x)$ is the atomic measure on $\mathbb{R}^n$ supported on the lattice $(a\mathbb{Z})^n$ with

\begin{equation}
\label{eq:interpolate_bandlimited}
\rho_n'(\{k\}) = \int_{\mathbb{R}^n} \prod_{i=1}^n S\Big(\frac{x_i-k_i}{a}\Big) \, d\rho_n(x), 
\end{equation}
for all  $k=(k_1, k_2, ..., k_n)  \in (a\mathbb{Z})^n$. (There exist measures $\rho_n$ such that the integral in \eqref{eq:interpolate_bandlimited} will not converge, but for all $\rho_n$ which we consider this integral will indeed converge.)

Nonetheless  the existence of a measure $\rho_n'$ supported on $(a\mathbb{Z})^n$ satisfying \eqref{eq:sampling_correlations} is only a necessary condition that there exist   a point process having  such correlation measures. As we will see,  the measures defined by \eqref{eq:sampling_correlations} are sometimes realized as correlation measures of a point process, but sometimes they are not. The bandwidth $\frac{1}{2a}$ for the lattice $a\mathbb{Z}$ nonetheless retains a certain importance, and  we use the convention that the bandwidth $B = \frac{1}{2a}$ is called \textbf{the Nyquist bandwidth} (for mimicry on $a \mathbb{Z}$), following a naming convention in sampling theory. (More often $\frac{1}{a} = 2B$ is termed the {\bf Nyquist rate} (measured in samples per second) for sampling band-limited  functions whose Fourier transform has maximum frequency $B$ on a lattice with spacing $a\mathbb{Z}$.)

We note that there are general mathematical results asserting that for (stable) reconstruction of an arbitrary band-limited signal on  $\mathbb{R}^n$ with frequencies confined  to  a finite set of intervals having  measure $2B$ using a sampling scheme on $a\mathbb{Z}^n$, one must have $B \le \frac{1}{2a}$, see Landau \cite[Theorem 1]{Lan67a}, \cite{Lan67b}, who noted that the special case of an interval $[-B, B]$ was originally due to A. Beurling.

\subsection{Point Processes studied}

In this paper we will treat in detail the mimicry tradeoff between $a$ and $B$ for two particular point processes for which mimicry occurs: the Poisson process and the sine-process. 

\subsubsection{Poisson point process}
The Poisson process is in fact a family of point processes indexed by a parameter $\lambda>0$ called the {\em intensity}. \textbf{The Poisson process of intensity $\lambda$} may be characterized as follows \cite[Ex. 2.5]{Jo02}: it is the unique point process $w$ with correlation measures defined by 
\begin{equation}
\label{eq:poisson_def}
\Ee \distinctsum \phi(w_{j_1},...,w_{j_n}) = \int_{\mathbb{R}^n} \phi(x_1,...,x_n) \cdot \lambda^n\, dx_1\cdots dx_n,
\end{equation}
for all $n\geq 1$ and for all $\phi \in C_c(\mathbb{R}^n)$. Thus its $n$-point correlation measure is $d\rho_n(x) = \lambda^n d^n x$. From this fact and \eqref{eq:factorial_moments} it is easy to see for any $\lambda$ that the Poisson point process of intensity $\lambda$ has uniform local moments.

\subsubsection{Sine process}
The sine process is a  name often used  for the determinantal point process associated to  the  sine kernel  $K(x,y) = \frac{\sin \pi(x-y)}{\pi(x-y)}$ for $x \ne y$, and $K(x,x) \equiv 1$, cf. \cite{BoDIK18}). \textbf{The sine process} may be characterized as follows \cite[Ch. 4]{HoKrPeVi09}: it is the unique point process $z$ with correlation measures $d\rho_n(x) = \det_{n\times n}[S(x_i-x_j)]\, d^n x$:
\begin{equation}
\label{eq:sinekernel_def}
\Ee \distinctsum \phi(z_{j_1},...,z_{j_n}) = \int_{\mathbb{R}^n} \phi(x_1,...,x_n) \cdot \det_{n\times n}\Big[S(x_i-x_j)\Big]\, dx_1\cdots dx_n,
\end{equation}
for all $n\geq 1$ and all $\phi \in C_c(\mathbb{R}^n)$. Here $\det_{n\times n}[\cdot]$ denotes an $n\times n$ determinant, and $S(x)$ is the sinc function given in \eqref{eq:sinc_def}. Furthermore, by convention, the right hand side of \eqref{eq:sinekernel_def} for $n=1$ has the meaning $\int_\mathbb{R} \phi(x_1)\, dx_1$. From this correlation measure and \eqref{eq:factorial_moments} it is easy to see that the  sine process also has uniform local moments.

\subsection{Main results: general processes}
\label{subsec:mainresults}

We prove two general results about mimicry. The first  result is a uniqueness result for the correlation functions of a  mimicking process strictly above  the Nyquist bandwidth.

\begin{thm}
\label{thm:nyquist-dichotomy-1}
For any u.l.m. point process $u$ on $\mathbb{R}$, if there exists a  point process $u'$ supported on $a\mathbb{Z}$ that mimics $u$ at the bandwidth 
$[-B, B]$  with $B > \frac{1}{2a}$ then its $n$-point correlation measures, supported on $(a \mathbb{Z})^n$, are uniquely determined for all $n \ge 1$.
\end{thm}

The uniqueness assertion of Theorem \ref{thm:nyquist-dichotomy-1} need not hold   for $B\leq\tfrac{1}{2a}$. In fact, for all $a > 0$ and $B=\tfrac{1}{2a}$, 
there exist two distinct point processes supported on $a\mathbb{Z}$, with different correlation measures for all $n\geq 1$, which mimic the Poisson process (of any intensity $\lambda$). See Proposition \ref{prop:therearetwo}. 

We prove Theorem \ref{thm:nyquist-dichotomy-1} in Section \ref{sec:2}, where we give a reconstruction formula for these correlation measures in Theorem \ref{thm:nyquist_mimickry}. 
Note that $n$-point correlation functions do not always uniquely determine a point process, but they do so provided a suitable bound on the growth of the local moments of the process holds.

The second result gives an upper bound for the mimicry tradeoff for translation invariant u.l.m. point processes. We  call a point process  {\em $\mathbb{R}$-translation invariant (in the correlation sense)} if for all $n \ge 1$ its $n$-point correlation measures satisfy for each $(x_1, x_2, ..., x_n) \in \mathbb{R}^n$, 
\begin{equation}
\label{eq:translation_invariant} 
\rho_n(x_1 +t, x_2 +t , \cdots , x_n + t) = \rho_n(x_1, x_2 , \cdots , x_n) \quad \mbox{for all} \quad t \in \mathbb{R}.
\end{equation}
The usual notion of translation-invariance for a point process  requires that  its probability law be  invariant in distribution  under translations, compare \cite[Sect. 4.2.6]{AnGuZe10}. This notion implies translation-invariance in the correlation sense. For u.l.m. point processes with a suitable growth bound on their local moments, so that the correlation functions uniquely determine the law of the process, the two definitions are equivalent.

We have a similar notion of  {$(a \mathbb{Z)}$-translation invariance} (in the correlation sense), for point processes supported on the lattice $a\mathbb{Z}$
restricting $(x_1, x_2, \cdots, x_n) \in (a \mathbb{Z})^n$ and $t \in a\mathbb{Z}$ above. Furthermore we say a point process $u$ is \emph{trivial} if for all Borel subsets $B$, $\#_B(u) = 0$ almost surely. A point process is said to be \emph{non-trivial} otherwise.


\begin{thm}
\label{thm:upper_bound_nyquist}
Let  $u$ be  a non-trivial u.l.m. point process on $\mathbb{R}$ that is $\mathbb{R}$-translation-invariant in the correlation sense. If a point process $u'$ supported on $a\mathbb{Z}$  mimics $u$ at the bandwidth $[-B, B]$, then necessarily $B \le \frac{1}{a}$. That is, for any lattice $a \mathbb{Z}$ the process $u$ cannot be mimicked on $a\mathbb{Z}$ above twice its Nyquist bandwidth.
\end{thm}

The bound $\frac{1}{a}$ of Theorem \ref{thm:upper_bound_nyquist} is tight. We show in Theorem \ref{thm:fullrange_poisson} that it is attained for  the Poisson process.

Theorem \ref{thm:upper_bound_nyquist} is derived at the end of  Section \ref{sec:2} using a result that if $u$ is a translation-invariant point process (in the correlation sense) that can be mimicked by a process $u'$ on $a\mathbb{Z}$ at a bandwidth $ B >\frac{1}{2a}$ above the Nyquist bandwidth, then $u'$ is necessarily  $a \mathbb{Z}$-translation invariant (in the correlation sense). We obtain a contradiction from this property if $B> \frac{1}{a}$.

\subsection{Main results:  Poisson process and sine process}

We prove specific results for the Poisson process and the sine process. For each lattice spacing $a$, one may ask about the full range of bandwidths $B$ for which the point process can be mimicked. 

For the Poisson process we have a complete characterization: mimicry is possible for $B$ up to and including twice the Nyquist bandwidth.

\begin{thm}[Poisson process mimicry - general bandwidth]
\label{thm:fullrange_poisson}
Let $\lambda > 0$ be arbitrary. We have,
\begin{enumerate}
\item For all $a > 0$, if $ B \leq \frac{1}{a}$, then the Poisson process with intensity $\lambda$ can be mimicked at bandwidth $[-B, B]$ by a u.l.m. point process supported on $a\mathbb{Z}$.
\item For all $a > 0$, if $ B > \frac{1}{a}$, then the Poisson process with intensity $\lambda$ cannot be mimicked at bandwidth $[-B, B]$ by a u.l.m. point process supported on $a\mathbb{Z}$.
\end{enumerate}
\end{thm}

\begin{cor}[Poisson process mimicry - Nyquist bandwidth]
\label{thm:Nyquist_poisson}
Let $\lambda$ be arbitrary. For each $a > 0$, the Poisson process with intensity $\lambda$ can be mimicked at the Nyquist bandwidth $[-\frac{1}{2a},\frac{1}{2a}]$ by a u.l.m. point process supported on $a\mathbb{Z}$.
\end{cor}

We  have stated this corollary for mimicry of the Poisson process at the Nyquist bandwidth in order to compare with Corollary \ref{cor:Nyquist_sinekernel} for the sine process, which exhibits a very different behavior in this regime.

Turning to the sine process, we obtain  a partial answer on bandwidths when mimicry is possible, for a general sampling lattice $a\mathbb{Z}$.

\begin{thm}[Sine process mimicry - general bandwidth]
\label{thm:fullrange_sinekernel}
We have,
\begin{enumerate}
\item For all $0 < a \leq 1$, if $B \leq \frac{1-a}{a} $, then the sine process can be mimicked at bandwidth $[-B, B]$ by a u.l.m. point process supported on $a\mathbb{Z}$.
\item For all  $0< a \leq \frac{1}{2}$, if $B > \frac{1-a}{a}$, then the sine process cannot be mimicked at bandwidth $[-B, B]$ by a u.l.m. point process supported on $a\mathbb{Z}$.
\item If $a > \frac{1}{2}$ and $B \geq \frac{1}{2a}$, then the sine process cannot be mimicked at bandwidth $[-B, B]$ by a u.l.m. point process supported on $a\mathbb{Z}$.
\end{enumerate}
\end{thm}

This result gives a complete answer  for the sine process at the Nyquist bandwidth. An important feature of the answer is that existence of mimicry at the Nyquist bandwidth depends on the value of $a$. 

\begin{cor}[Sine process mimicry - Nyquist bandwidth]
\label{cor:Nyquist_sinekernel}
The sine  process can  be mimicked  by a u.l.m. point process  supported on $a\mathbb{Z}$ at the Nyquist bandwidth $[-\frac{1}{2a}, \frac{1}{2a}]$ if and only if $0 < a \leq \frac{1}{2}$.
\end{cor}

\begin{figure}[t]
    \centering
    \begin{minipage}{0.48\textwidth}
        \centering
        \includegraphics[width=1\textwidth]{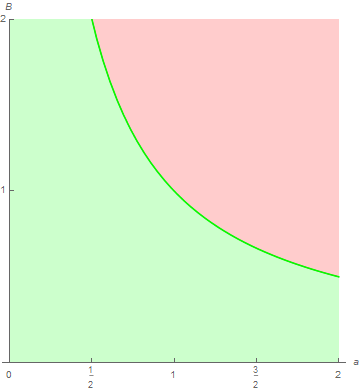} 
        \title{Mimicry for the Poisson process.}
    \end{minipage}\hfill
    \begin{minipage}{0.48\textwidth}
        \centering
        \includegraphics[width=1\textwidth]{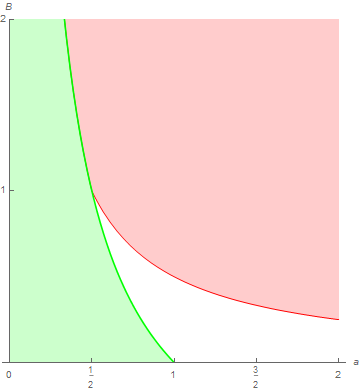} 
        \title{Mimicry for the sine  process.}
    \end{minipage}
	\caption{A plot of the regions $(a,B)$ for which the Poisson process and sine  process can be mimicked at bandwidth $B$ by a point process with uniform moments supported on $a\mathbb{Z}$. In the green region these point processes can be mimicked, while in the red region they cannot. In the white region of the second plot we currently have no information.}
\label{fig:GoNoGoPlots}
\end{figure}

This answer contrasts with the Poisson process case, where mimicry is  possible at the Nyquist bandwidth for every $a >0$.

For $0< a < \frac{1}{2}$, Theorem \ref{thm:fullrange_sinekernel} implies mimicry is possible slightly beyond the Nyquist bandwidth, by an amount depending on $a$. For $a > \frac{1}{2}$ we do not determine the complete range of mimicry, however Theorem \ref{thm:fullrange_sinekernel} shows that the  mimicry range is strictly below the Nyquist bandwidth.

We note that the sine process at $a=\frac{1}{2}$ is the largest value of $a$ where the Nyquist bandwidth can be achieved. This process plays an important role in \cite{LaRo19}.

The regions of $a,B$ spelled out by these theorems are plotted in Figure \ref{fig:GoNoGoPlots}. It would be very interesting to 
understand those $a,B$ not described by Theorem \ref{thm:fullrange_sinekernel}, left white in Figure \ref{fig:GoNoGoPlots}.

\subsection{An application to the Alternative Hypothesis}
\label{subsec:alterhyp}

The questions treated in this paper were motivated by a problem originating in number theory regarding zeros of the Riemann zeta function. We treat this problem in a companion paper \cite{LaRo19}, and give a brief description here.

Let the nontrivial zeros of the Riemann zeta function in the upper half-plane be listed as $\{ \beta_k + i \gamma_k \}_{k \in \mathbb{Z} } $ in increasing order of ordinate, taking
$0< \gamma_1 \le \gamma_2 \le \gamma_3 \cdots$. We define the rescaled zeta zero ordinates
$$
\tilde{\gamma}_k := \tfrac{1}{2\pi} \gamma_k \log \gamma_k.
$$
It is known that the $\tilde{\gamma_k}$ have on average a spacing  of $1$ between consecutive values.  (This result goes back to Riemann's original paper \cite{Ri1859}, for a proof see \cite[Corollary 14.2]{MoVa07}.) The \emph{Alternative Hypothesis} refers to the (seemingly outlandish) supposition that the spacings $\tilde{\gamma}_{k+1}-\tilde{\gamma}_k$ always lie approximately in the set $\tfrac{1}{2}\mathbb{Z}$. It is discussed in a 2004 AIM note \cite{AIM04}, Farmer, Gonek and Lee \cite[Section 2]{FarGonLee14} and in Baluyot \cite{Ba16}.

The Alternative Hypothesis is of special interest because of known connections between the spacings of zeros of the zeta function and the existence of Landau-Siegel zeros (see e.g. Conrey and Iwaniec \cite{CoIw02}). The Alternative Hypothesis is expected to be false, and indeed it is contradicted by the well-known GUE Hypothesis, 
that the spacing between zeros of the zeta function follow  a distribution  coming from  random matrix theory, concerning rescaled eigenvalues of the Gaussian Unitary Ensemble,  cf. \cite{Mo73},  \cite{Od87}, \cite{KaSa99}.  On the other hand, the GUE Hypothesis remains a conjecture, even assuming the Riemann Hypothesis, and it is natural to ask whether the Alternative Hypothesis can be ruled out just by \emph{what is known about the statistical distribution of zeros of the zeta function}. By this we mean the  known information about $n$-level correlation functions of zeros that was proved by Rudnick and Sarnak \cite{RuSa96} for all $n \ge 1$, extending results for $n=2$ and $n=3$ (\cite{Mo73}, \cite{He94}).

Rudnick and Sarnak characterized the correlation functions of zeros against certain band-limited test functions; their result amounts to knowing just a bit less than the assertion that the renormalized zeros mimic the sine process at a bandwidth $B=1$. In the companion paper \cite{LaRo19} we review an exact statement of their result, and using ideas related to those in this paper, we show that the Alternative Hypothesis cannot be ruled out by what is known about the statistical distribution of zeros of the zeta function. This is done via the construction of a counterexample Alternative Hypothesis point process which uses the $1/2$-discrete sine process in its construction.

Recently Tao has independently treated the Alternative Hypothesis (using slightly different methods) in a blog post \cite{Ta19}. He constructs an alternate distribution $ACUE$  for  eigenvalues of unitary matrices $U(N)$; in  a suitable scaling limit as $N \to \infty$ his  construction yields the Alternative Hypothesis point process treated here. 

The present paper investigates conditions permitting mimicking by a lattice process $a\mathbb{Z}$  in greater generality than  \cite{LaRo19}. In particular, Corollary \ref{cor:Nyquist_sinekernel} reveals that  the ability  to construct a counterexample Alternative Hypothesis point process depends upon quite special properties of the sine process and the lattice-bandwidth combination $(a,B) = (1/2,1)$. In particular (see Figure \ref{fig:GoNoGoPlots}), the point $(a,B) = (1/2,1)$ occurs on the boundary of mimicry for the sine process, and even a slight perturbation off this lattice spacing or bandwidth would no longer allow for it.

We note that while the Band-Limited Mimicry Problem as posed above seems natural from the perspective of both applications to number theory and what one is able to say about it, one may reasonably ask broader questions. For instance, one may generalize Definition \ref{dfn:bandlimited_mimicry} so that \eqref{eq:bandlimited_equality} holds for a different collection of functions $\eta$ than those with Fourier transform supported on $[-B,B]^n$ (e.g. one might allow $\eta$ to be bandlimited in $[-B_n,B_n]^n$ for constants $B_n$ which vary with $n$). It would be interesting to see if a more general theory along these lines can be developed, but we do not pursue this here.

\section{The Nyquist bandwidth}
\label{sec:2}
The Nyquist bandwidth has an important implications regarding correlation measures. In this section we prove  for $B$ strictly larger than the Nyquist bandwidth $\frac{1}{2a}$ that  all the correlation measures of any u.l.m. mimicking discrete point process on $a \mathbb{Z}$ are uniquely determined. This result does not address the question whether any such mimicking  discrete point process exists. We then study translation invariance (in the correlation sense) and  deduce that $\mathbb{R}$ translation invariant point processes cannot be mimicked above twice the Nyquist bandwidth.


\subsection{Uniqueness of correlation functions above the Nyquist bandwidth}
In what follows for $0 < \varepsilon < 1/2$, we let $\beta_\varepsilon$ be an even  bump function with the following four properties:
\begin{equation}
\label{eq:beta_0}
0 \leq \beta_\epsilon(\xi) \leq 1, \quad \textrm{for all}\, \xi \in \mathbb{R},
\end{equation}
\begin{equation}
\label{eq:beta_1}
\beta_\varepsilon(\xi) = 1,\quad \textrm{for}\, |\xi|\leq 1/2-\epsilon,
\end{equation}
\begin{equation}
\label{eq:beta_2}
\beta_\varepsilon(\xi) = 0,\quad \textrm{for}\, |\xi|\geq 1/2+\epsilon,
\end{equation}
\begin{equation}
\label{eq:beta_3}
\beta_\epsilon(\tfrac{1}{2}+x) = 1 - \beta_\epsilon(\tfrac{1}{2} - x),\quad \textrm{for all}\, 0 \leq x < 1/2.
\end{equation}
A `bump function' is any function that is $C^{\infty}$-smooth and compactly supported. We omit the  details in constructing such  bump functions, see Lee \cite[Lemma 2.22]{Lee13}. The function  $\beta_\epsilon(\xi)$ should be seen as a smooth approximation to the indicator function of the interval from $[-\tfrac{1}{2},\tfrac{1}{2}]$. Note further that the functions $...,\beta_\epsilon(\xi-1), \beta_\epsilon(\xi), \beta_\epsilon(\xi+1),...$ form a partition of unity for the real line.

\begin{thm}
\label{thm:nyquist_mimickry}
If a u.l.m. point process  $u$ can be mimicked at bandwidth $B$ by a  point process $u'$ supported on $a\mathbb{Z}$, and if $B > \frac{1}{2a}$, then the correlation measures $\rho'_n$  of $u'$, which are supported on $(a\mathbb{Z})^n$, are uniquely determined and satisfy
\begin{equation}
\label{eq:above_nyquist}
\rho'_n(\{ (k_1, k_2, \cdots, k_n) \}) = \int_{\mathbb{R}^n} \prod_{i=1}^n \hat{\beta}_\varepsilon\Big(\frac{x_i-k_i}{a}\Big)\, d\rho_n(x)\quad \textrm{for all} \; k \in (a\mathbb{Z})^n,
\end{equation} 
for any bump function $\beta_{\varepsilon}$ satisfying  \eqref{eq:beta_0} -\eqref{eq:beta_3} and for all sufficiently small $\varepsilon$ (where sufficiently small depends on $B$).
\end{thm}

\begin{remark}
{\rm 
For $\varepsilon > 0$, the function $\hat{\beta}_\varepsilon(x)$ is a Schwartz function and the integral \eqref{eq:above_nyquist} converges due to the assumption of uniform local moments on the point process $u$.
}
\end{remark}

\begin{proof} 
We begin by showing that for $x \in (a\mathbb{Z})^n$,
\begin{equation}
\label{eq:beta_lattice}
\prod_{i=1}^n \hat{\beta}_\varepsilon\Big(\frac{x_i-k_i}{a}\Big) = \mathbf{1}_k(x).
\end{equation}
Note that
\begin{equation}
\label{eq:beta_fourier}
\prod_{i=1}^n \hat{\beta}_\varepsilon\Big(\frac{x_i-k_i}{a}\Big) = \prod_{i=1}^n \int_\mathbb{R} \beta_\varepsilon(\xi) e(\xi(x_i-k_i)/a)\, d\xi.
\end{equation}
For fixed $i$, if  $(x_i-k_i)/a \in \mathbb{Z}$, then $f(\xi)= e(\xi(x_i-k_i)/a)$ has period $1$ and so using the properties \eqref{eq:beta_1}, \eqref{eq:beta_2}, and \eqref{eq:beta_3},
\begin{align}
\label{eqn:one-variable}
\int_\mathbb{R} \beta_\varepsilon(\xi) e(\xi(x_i-k_i)/a)\, d\xi
&= \int_{-3/2}^{3/2} \beta_\varepsilon(\xi)e(\xi(x_i-k_i)/a)\, d\xi \nonumber\\
& = \int_{-1/2}^{1/2} [\beta_\varepsilon(\xi) + \beta_\varepsilon(-1+\xi) + \beta_\varepsilon(1+\xi)] e(\xi(x_i-k_i)/a)\, d\xi \nonumber\\
&= \int_{-1/2}^{1/2} 1 \cdot e(\xi(x_i-k_i)/a)\, d\xi  \nonumber\\
&= \mathbf{1}_{k_i}(x_i).
\end{align}
Applying this formula for each $i$  in \eqref{eq:beta_fourier} yields \eqref{eq:beta_lattice}. (Note that the equality \eqref{eqn:one-variable}  does not hold if $k_i \in \mathbb{R} \smallsetminus a\mathbb{Z}$.)

Let $\eta \in \mathcal{S}(\mathbb{R}^n)$ denote
\begin{equation}
\label{eq:Nyquist_eta}
\eta(x) := \prod_{i=1}^n \hat{\beta}_\varepsilon\Big(\frac{x_i-k_i}{a}\Big).
\end{equation}
We have
$$
\hat{\eta}(\xi) = a^n e(-k\cdot \xi) \prod_{i=1}^n \beta_\varepsilon(a \xi_i).
$$
which is supported in $[ -\frac{1}{2a} - \frac{\epsilon}{a}, \frac{1}{2a} + \frac{\epsilon}{a}]^n$. 
If $B > 1/(2a)$, then for sufficiently small $\varepsilon >0$ we have $\supp\, \hat{\eta} \subset [-B,B]^n$.

From the support  of $u'$ falling in $a\mathbb{Z}$, and from \eqref{eq:beta_lattice}  we have
\begin{eqnarray}
\label{eq:rho-formula}
\rho_n'(\{k\}) &= &\mathbb{E} \distinctsum \mathbf{1}_{k} (u_{j_1}',...,u_{j_n}') \nonumber \\
&=& \mathbb{E} \distinctsum \eta(u_{j_1}',...,u_{j_n}').
\end{eqnarray}
If $u$ is mimicked at bandwidth $B$ by $u'$, since $\supp\, \hat{\eta} \subset [-B,B]^n$ we have
$$
\mathbb{E} \distinctsum \eta(u_{j_1}',...,u_{j_n}') =\mathbb{E}\distinctsum \eta(u_{j_1}, ..., u_{j_n}) = \int_{\mathbb{R}^n} \eta(x) d\rho_n(x),
$$
where the second equality holds by Proposition \ref{prop:toSchwartz}. Combining this equality with \eqref{eq:rho-formula} gives
$$
\rho_n'(\{k\}) =\int_{\mathbb{R}^n} \eta(x) d\rho_n(x) = \int_{\mathbb{R}^n}  \prod_{i=1}^n \hat{\beta}_\varepsilon\Big(\frac{x_i-k_i}{a}\Big)\, d\rho_n(x).
$$
as asserted. 
\end{proof}

Theorem \ref{thm:nyquist-dichotomy-1} is a direct consequence of Theorem \ref{thm:nyquist_mimickry}.

\subsection{Translation-invariant point processes and  Nyqist bandwidth}
\label{subsec:21}

Recall from Section \ref{subsec:mainresults} that a point process on $\mathbb{R}$ is {\em $\mathbb{R}$ translation-invariant} (in the correlation sense) if every $n$-level correlation function is translation invariant: For all $n \ge 1$,
$$
\rho_n(x_1, x_2, ..., x_n)= \rho_n(x_1+t, x_2+t , ..., x_n+t) \quad \mbox{for all} \quad  t \in \mathbb{R}.
$$
A point process  is {\em $a \mathbb{Z}$ translation-invariant } (in the correlation sense)  if every $n$-level correlation function is translation invariant:
$\rho_n(x_1, x_2, ..., x_n)= \rho_n(x_1+t, x_2+t , ..., x_n+t)$ for all $t \in a\mathbb{Z}$.

\begin{cor}
\label{thm:nyquist_translation_inv_bound}
Let $u$ be a  u.l.m. point process that is $\mathbb{R}$ translation-invariant in the correlation sense. Suppose that $u$ can be mimicked by a  point process $u'$ supported on $a\mathbb{Z}$ at bandwidth $B$ with  $B > \frac{1}{2a}$. Then $u'$ is  $a \mathbb{Z}$ translation-invariant in the correlation sense. That is, for each $n \ge 1$ the (uniquely determined)  correlation measure $\rho'_n$  of $u'$, which is supported on $(a\mathbb{Z})^n$, is $a \mathbb{Z}$-translation invariant.
\end{cor}

\begin{proof}
Since $B > \frac{1}{2a}$, by Theorem \ref{thm:nyquist_mimickry} the correlation functions of the process $u'$ are uniquely determined. The $a \mathbb{Z}$-translation invariance of all the correlation functions $\rho^{'}_n(k_1, k_2, ..., k_n)$ of $u'$ then follows from \eqref{eq:above_nyquist}. In more detail: we have, for each $(k_1, k_2, ..., k_n) \in (a\mathbb{Z})^n$ and any translation $k \in a\mathbb{Z}$, 
\begin{eqnarray*}
\label{eq:abov_-translation_invariance}
\rho'_n(\{ (k_1-k, \cdots, k_n -k)\} ) &= & \int_{\mathbb{R}^n} \prod_{i=1}^n \hat{\beta}_\varepsilon\Big(\frac{x_i-k_i+k}{a}\Big)\, d\rho_n(x_1, \cdots , x_n)\\
&=&  \int_{\mathbb{R}^n} \prod_{i=1}^n \hat{\beta}_\varepsilon\Big(\frac{y_i-k_i}{a}\Big)\, d\rho_n(y_1-k, \cdots, y_n-k)\\
&=&  \int_{\mathbb{R}^n} \prod_{i=1}^n \hat{\beta}_\varepsilon\Big(\frac{y_i-k_i}{a}\Big)\, d\rho_n(y_1, \cdots, y_n)\\
&= & \rho'_n(k_1, \cdots, k_n),
\end{eqnarray*}
with the third equality holding by $\mathbb{R}$-translation invariance of the $n$-point correlation function  $\rho_n$ of $u$, and the first and last inequality hold (for  $(k_1-k, \cdots , k_n-k)  \in (a\mathbb{Z})^n$) by \eqref{eq:above_nyquist}. (Actually only the $a\mathbb{Z}$-translation invariance of $\rho_n$ is needed for the third equality to hold.)   
\end{proof}

\subsection{ Proof of  Theorem \ref{thm:upper_bound_nyquist}.}
\label{subsec:23}

\begin{proof} [Proof of Theorem \ref{thm:upper_bound_nyquist}.]
We suppose that there exists a non-trivial process $u'$ supported on $ a\mathbb{Z}$ that mimics $u$ to bandwidth $B > \frac{1}{a}$ and  obtain a contradiction. By the result of Theorem \ref{thm:nyquist_translation_inv_bound} the process $u'$ is $a\mathbb{Z}$-translation invariant in the correlation sense.

On the other hand, letting $a'=\frac{1}{2}a$, we have that $u'$ is also supported on the lattice $a'\mathbb{Z} = \frac{1}{2}a\mathbb{Z}$, as this includes $a\mathbb{Z}$ as a sublattice. But the process $u'$ mimics $u$ to bandwidth $B> \frac{1}{a} = \frac{1}{2a'},$ which is above the Nyquist bandwidth for the lattice $a'\mathbb{Z}$. Therefore Theorem \ref{thm:nyquist_translation_inv_bound} applies to say that this process $u'$ must be $a'\mathbb{Z}$-translation invariant in the correlation sense.

However $u'$ is manifestly not $a'\mathbb{Z}$ translation-invariant in the correlation sense, because it is supported on $a\mathbb{Z}$, a lattice which does not include the point $a'$. Indeed, because $u'$ is non-trivial and translation invariant in the correlation sense on $a\mathbb{Z}$, we must have $\Ee \#_{\{0\}}(u') > 0$. Then translation invariance in the correlation sense on $a'\mathbb{Z}$ implies $\Ee \#_{\{a'\}}(u') > 0$, which cannot be the case if $u'$ is supported on $a\mathbb{Z}$. 
\end{proof}

We note that Theorem \ref{thm:upper_bound_nyquist} is not true if the assumption of translation invariance is dropped. Indeed, consider the point process which consists of a single point located at the position $0$ almost surely. Since for any $a$ this point process is already itself supported on the lattice $a\mathbb{Z}$, mimicry occurs for any parameters $(a,B)$.

\section{Mimicry of the Poisson process}
\label{sec:poisson}

\subsection{The discrete Poisson process}
\label{subsec:discretePoisson}
In this section we prove Theorem \ref{thm:fullrange_poisson}, describing when the Poisson process can be mimicked.

It ends up that in the range of $a, B$ where the process can be mimicked, it is mimicked just by the discrete Poisson process.

\begin{dfn}
\label{def:discretePoisson}
For any $a>0$ and any $\lambda>0$ the \textbf{discrete Poisson process on $a\mathbb{Z}$ of intensity $\lambda$} is the point process $w^\ast=(w_j^{\ast})_{j \in \mathbb{Z}}$ such that for each $k \in a\mathbb{Z}$, the number of points at each site $\#_k(w^\ast)$ are independent and identically distributed random variables, with each variable a Poisson random variable with mean $a\lambda$.
\end{dfn}

The discrete Poisson process on $a\mathbb{Z}$ of intensity $\lambda$  is never a simple point process.

\begin{prop}
\label{prop:discretePoisson_corr}
Letting $w^\ast$ be the discrete Poisson process on $a\mathbb{Z}$ of intensity $\lambda$, we have for all $n\geq 1$ and $\phi \in \mathcal{S}(\mathbb{R}^n)$,
$$
\Ee \distinctsum \phi(w_{j_1}^\ast,..., w_{j_n}^\ast) = \sum_{k \in (a\mathbb{Z})^n} (a\lambda)^n \phi(k).
$$
\end{prop} 
\begin{proof} 
This follows from the independence of the random variables $\#_k(w^\ast)$ for different $k$, and the fact that the factorial moments of Poisson random variables satisfy
$$
\mathbb{E}\, \#_k(w^\ast) (\#_k(w^\ast)-1)\cdots (\#_k(w^\ast)-(m-1)) = (a\lambda)^m.
$$
\end{proof}

\subsection{Mimicry for $B \leq \frac{1}{a}$, no mimicry otherwise}
\label{subsec:Poisson_existence}
We now show the first part of Theorem \ref{thm:fullrange_poisson}, that the Poisson process can be mimicked by the discrete Poisson process. The proof depends on the Poisson summation formula, which we recall in a suitable form.

\begin{thm}[Poisson summation formula]
\label{thm:poisson_summation}
For all $\phi \in \mathcal{S}(\mathbb{R}^n)$,
$$
a^n \sum_{k \in (a\mathbb{Z})^n} \phi(k) = \sum_{j \in (a^{-1}\mathbb{Z})^n} \hat{\phi}(j).
$$
\end{thm}

\begin{proof} 
The usual formulation of Poisson summation states this for $a=1$ (see \cite[Theorem 3.1.17]{Gr14}): $\sum_{k\in \mathbb{Z}} \phi(k) = \sum_{j \in \mathbb{Z}} \hat{\phi}(j)$. Replacing $\phi(x)$ with $a^n \phi(ax)$ yields the result for general $a$.
\end{proof}

\begin{proof}[Proof of Theorem \ref{thm:fullrange_poisson}, part (i)]
We show that for $B \leq 1/a$, the Poisson process with intensity $\lambda$ is mimicked at bandwidth $[-B,B]$ by the discrete Poisson process on $a\mathbb{Z}$ with intensity $\lambda$. For $\eta \in \mathcal{S}(\mathbb{R}^n)$ with $\supp \hat{\eta} \subset [-B,B]^n$, we must show that
$$
\Ee \distinctsum \eta(w_{j_1},...,w_{j_n}) = \Ee \distinctsum \eta(w_{j_1}^{\ast},...,w_{j_n}^{\ast}).
$$
Using \eqref{eq:poisson_def} for the Poisson process and Proposition \ref{prop:discretePoisson_corr} for the discrete Poisson process, this requires the equality
\begin{equation}
\label{eq:poisson_verify}
\int_{\mathbb{R}^n} \eta(x) \lambda^n\, d^n x = \sum_{k \in (a\mathbb{Z})^n} (a\lambda)^n \eta(k).
\end{equation}
The left side is $\lambda^n \hat{\eta}(0)$. Using  Poisson summation the right side is 
$$
(a\lambda)^n \sum_{k \in (a\mathbb{Z})^n} \eta(k) = \lambda^n \sum_{j \in (a^{-1} \mathbb{Z})^n} \hat{\eta}(j)= \lambda^n \hat{\eta}(0),
$$
where the last equality holds because  $\supp \hat{\eta} \subset [-\frac{1}{a}, \frac{1}{a}]^n$, since $B \le \frac{1}{a}$.
Since $\hat{\eta}$  is a Schwartz function it necessarily must vanish at all points on the boundary of its support, hence the only non-vanishing point $k$ in $(a^{-1}\mathbb{Z})^n$ is $k=0$.
\end{proof}

The other half of Theorem \ref{thm:fullrange_poisson} follows from results we have already proved:

\begin{proof}[Proof of Theorem \ref{thm:fullrange_poisson}, part (ii)]
This is a direct consequence of Theorem \ref{thm:upper_bound_nyquist}.
\end{proof}

As we have mentioned in the context of Theorem \ref{thm:nyquist-dichotomy-1}, the mimicry demonstrated above need not be unique outside the range $B > \tfrac{1}{2a}$.

\begin{prop}
\label{prop:therearetwo}
For any $\lambda > 0$ and any $a$ and $B$ satisfying $0 < B \leq \tfrac{1}{2a}$, there exist two distinct point processes supported on $a\mathbb{Z}$ which mimic the Poisson process of intensity $\lambda$, and have different correlation measures for all $n\geq 1$ . 
\end{prop}

\begin{proof}
Let $w^\ast$ be the discrete Poisson process on $a\mathbb{Z}$ with intensity $\lambda$ and let $w^{\ast\ast}$ be the discrete Poisson process on $2a\mathbb{Z}$ with intensity $\lambda$. For $B\leq \tfrac{1}{2a}$, we have that both $w^\ast$ and $w^{\ast\ast}$ mimic the Poisson process at bandwidth $[-B,B]$. For $w^\ast$, this is implied directly by Theorem \ref{thm:fullrange_poisson}. For $w^{\ast\ast}$, we also verify mimicry from Theorem \ref{thm:fullrange_poisson}, with the lattice spacing $a$ replaced by a lattice spacing of $2a$. Yet $2a\mathbb{Z} \subset a\mathbb{Z}$, so both $w^\ast$ and $w^{\ast\ast}$ are supported on the lattice $a\mathbb{Z}$, and it is plain from the definition that $w^\ast$ and $w^{\ast\ast}$ have different correlation measures for all  $n\geq 1$.
\end{proof}

\section{Mimicry of the sine process}
\label{sec:sinekernel}

\subsection{The discrete  sine process}
\label{subsec:discrete_sinekernel}
In this section we prove Theorem \ref{thm:fullrange_sinekernel}. A key tool will be the discrete sine process.

\begin{thm}
\label{thm:discrete_sinekernel_existence}
For each $0 < a \leq 1$, there exists a unique point process $z^\ast$ on $a\mathbb{Z}$ such that for all $n\geq 1$ and all $\phi \in \mathcal{S}(\mathbb{R}^n)$,
\begin{equation}
\label{eq:discrete_sinekernel_corrs}
\Ee \distinctsum \phi(z_{j_1}^\ast,...,z_{j_n}^\ast) = \sum_{k \in (a\mathbb{Z})^n} a^n \det_{n\times n}[ S(k_i - k_j)] \phi(k).
\end{equation}
Moreover $z^\ast$ has uniform local moments.
\end{thm}

\begin{dfn}
\label{dfn:discrete_sine_kernel_def}
The point process $z^\ast$ described by Theorem \ref{thm:discrete_sinekernel_existence} is called \textbf{the discrete sine process on $a\mathbb{Z}$.}
\end{dfn}

The discrete sine process is not new; in various guises it has appeared in \cite{BoOkOl00,Jo02,Wi72,Wi94} and a proof of its existence follows the same ideas as for the (continuous) sine process, coming from the theory of determinantal point processes. The details of this proof however do not seem to be in the literature. We provide a proof of Theorem \ref{thm:discrete_sinekernel_existence} in the appendix of a companion paper \cite{LaRo19}. For $a > 1$, there does not exist a point process with correlation structure described by \eqref{eq:discrete_sinekernel_corrs}, see \cite[Remark A.3]{LaRo19}.

The discrete sine process on $a \mathbb{Z}$  is  a simple point process for $0< a \le 1.$ (\cite[Proposition 4.4]{LaRo19}). This simplicity property exhibits  repulsion of points, a characteristic property of determinantal point processes, cf. \cite[Chap. 1]{HoKrPeVi09}.

\subsection{Mimicry for $B \le \frac{1-a}{a}$}
\label{subsec:sinekernel_existence}
We show that the sine process can be mimicked by the discrete sine process; this is the first part of Theorem \ref{thm:fullrange_sinekernel}. As in the previous section, our proof depends on Poisson summation.

\begin{proof}[Proof of Theorem \ref{thm:fullrange_sinekernel}, part (i)]
We show for $B \leq \tfrac{1-a}{a} = 1/a-1$, the sine process is mimicked by the discrete sine process on $a\mathbb{Z}$. By Theorem \ref{thm:discrete_sinekernel_existence} and \eqref{eq:sinekernel_def} this is just a matter of showing that for $\eta \in \mathcal{S}(\mathbb{R}^n)$ with $\supp \hat{\eta} \subset [-B,B]^n$,
\begin{equation}
\label{eq:sinekernel_verify}
\int_{\mathbb{R}^n} \eta(x) \det_{n\times n}[S(x_i-x_j)]\, d^n x = a^n \sum_{k \in (a\mathbb{Z})^n} \eta(k) \det_{n\times n}[S(k_i-k_j)].
\end{equation}
Let $g(x) = \eta(x) \det_{n\times n}[S(x_i-x_j)]$. Then \eqref{eq:sinekernel_verify} is just the claim that
$$
\int_{\mathbb{R}^n} g(x)\,d^n x = a^n \sum_{k \in (a\mathbb{Z})^n} g(k),
$$
and as the left hand side is $\hat{g}(0)$, this identity will be verified by Poisson summation if we show $\hat{g}(y) = 0$ whenever $y \notin (-1/a,1/a)^n$.

For notational reasons we let $E = [-1/2,1/2]$. One has the well-known computation
$$
S(x) = \int_\mathbb{R} \mathbf{1}_{E}(\xi) e(\xi)\, d\xi
$$
so, where $\mathfrak{S}_n$ is the symmetric group,
\begin{align*}
\det_{n\times n}[S(x_i-x_j)] &= \sum_{\sigma \in \mathfrak{S}_n} \sgn(\sigma) \prod_{i=1}^n S(x_i-x_j) \\
&= \sum_{\sigma \in \mathfrak{S}_n} \sgn(\sigma) \int_{E^n} e\Big(\sum_{i=1}^n \xi_i (x_i-x_{\sigma(i)})\Big)\, d^n\xi \\
&= \sum_{\sigma \in \mathfrak{S}_n} \sgn(\sigma) \int_{E^n} e\Big(\sum_{i=1}^n x_i (\xi_i - \xi_{\sigma^{-1}(i)})\Big)\, d^n\xi.
\end{align*}
Hence for $y \in \mathbb{R}^n$,
\begin{align}
\label{eq:ghat_eval}
\notag \hat{g}(y) &= \sum_{\sigma \in \mathfrak{S}_n} \sgn(\sigma) \int_{E^n} \int_{\mathbb{R}^n} e(-x\cdot y) e\Big(\sum_{i=1}^n x_i (\xi_i - \xi_{\sigma^{-1}(i)})\Big) \eta(x)\, d^nx\, d^n\xi \\
&= \sum_{\sigma \in \mathfrak{S}_n} \sgn(\sigma) \int_{E^n} \hat{\eta}(y_1 - (\xi_1 - \xi_{\sigma^{-1}(1)}), ..., y_n - (\xi_n - \xi_{\sigma^{-1}(n)}))\, d^n \xi.
\end{align}
But for $y \notin (-1/a,1/a)^n$, we must have $|y_i| \geq 1/a$ for some $i$, and hence for $\xi \in E^n$, we have $|y_i - (\xi_i - \xi_{\sigma^{-1}(i)})| \geq 1/a -1$. If $\hat{\eta}$ is supported in $[-B,B]^n$ with $B \leq 1/a-1$, we therefore have that the integrand in \eqref{eq:ghat_eval} vanishes for all $y \notin (-1/a,1/a)^n$, and $\hat{g}(y) = 0$ as we wanted. This therefore verifies \eqref{eq:sinekernel_verify} and proves the claim.
\end{proof}

\subsection{No mimicry for $a \leq \frac{1}{2}$ and $B > \frac{1-a}{a}$}
\label{subsec:sinekernel_nonexistence1}
We now prove part (ii) of Theorem \ref{thm:fullrange_sinekernel}. 
For $a \leq 1/2$, our strategy will be to suppose the sine  process can be mimicked for bandwidth $B > \frac{1-a}{a}$ and obtain a contradiction. Our main tool, as before, is Lemma \ref{thm:nyquist_mimickry}, but now we use $2$-level correlations.

\begin{proof}[Proof of Theorem \ref{thm:fullrange_sinekernel}, part (ii)]
Let $a\leq 1/2$ and let $z$ be the sine process. Suppose there exists a u.l.m. point process $z'$ supported on $a\mathbb{Z}$ which mimics $z$ at bandwidth $B > \tfrac{1-a}{a} = 1/a -1$; 
we will obtain a contradiction. 

For $a \leq 1/2$, this implies $B > 1/2a$ and so Theorem \ref{thm:nyquist_mimickry} applies. Thus for any $k\in (a\mathbb{Z})^2$ and all sufficiently small $\varepsilon > 0$,
\begin{align*}
\rho'_2(k) &= \int_{\mathbb{R}^2} \hat{\beta}_\varepsilon\Big(\frac{x_1 - k_1}{a}\Big) \hat{\beta}_\varepsilon\Big(\frac{x_2-k_2}{a}\Big)(1 - S(x_1-x_2)^2)\, dx_1 dx_2 \\
&= \int_{\mathbb{R}^2} \beta_\varepsilon(\xi_1) \beta_\varepsilon(\xi_2) e\Big(-\frac{k_1\xi_1+k_2\xi_2}{a}\Big) \Big[ \delta\Big(\frac{\xi_1}{a})\delta\Big(\frac{\xi_2}{a}\Big) - \delta\Big(\frac{\xi_1+\xi_2}{a}\Big)\Big(1 - \Big| \frac{\xi_1}{a}\Big|\Big)_+ \Big]\, d\xi_1 d\xi_2 \\
&= a^2 \Big( 1 - \int_\mathbb{R} \beta_\varepsilon(a\nu)^2 e((k_1-k_2)\nu) (1 - |\nu|)_+\, d\nu\Big),
\end{align*}
where the computation in the second line uses the Fourier pair $f(x) = S(x)^2$, $\hat{f}(\xi) = (1-|\xi|)_+$, and the computation in the third line makes use of the fact that $\beta_\varepsilon$ is even to simplify the resulting expression. As this is true for all sufficiently small $\varepsilon$, we can take the limit as $\varepsilon \rightarrow 0$, and see that
$$
\rho'_2(k) = a^2\Big(1 - \int_{-1/2a}^{1/2a} e((k_1-k_2)\nu) (1-|\nu|)_+\, d\nu\Big) = a^2 (1 - S(k_1-k_2)^2),
$$
with the last identity following because $(1-|\nu|)_+$ is supported in $[-1/2a,1/2a]$ for $a \leq 1/2$.

Hence for any $\eta \in \mathcal{S}(\mathbb{R})$, we must have for the point process $z'$,
\begin{equation}
\label{eq:sine_mimic_implication1}
\mathbb{E} \sum_{\substack{j_1, j_2 \\ \textrm{distinct}}} \eta(z_{j_1}', z_{j_2}') = a^2 \sum_{k \in (a\mathbb{Z})^2} \eta(k) (1 - S(k_1-k_2)^2).
\end{equation}

Yet if $z'$ mimics the sine process process at bandwidth $B$ for $\supp \hat{\eta} \subset [-B,B]^2$,
\begin{equation}
\label{eq:sine_mimic_implication2}
\mathbb{E} \sum_{\substack{j_1, j_2 \\ \textrm{distinct}}} \eta(z_{j_1}', z_{j_2}') = \int_{\mathbb{R}^2} \eta(x) (1 - S(x_1-x_2)^2)\, dx_1 dx_2.
\end{equation}
Let $g(x) = \eta(x)(1 - S(x_1-x_2)^2)$, so that as a consequence of \eqref{eq:ghat_eval} for $n=2$,
\begin{equation}
\label{eq:2point_ghat_eval}
\hat{g}(y_1,y_2) = \hat{\eta}(y_1,y_2) - \int_{\mathbb{R}} \hat{\eta}(y_1 - \xi, y_2 + \xi) (1-|\xi|)_+\, d\xi.
\end{equation}
By Poisson summation the expression on the right hand side of \eqref{eq:sine_mimic_implication1} is
$$
\sum_{j \in (a^{-1}\mathbb{Z})^2} \hat{g}(j),
$$
while the expression in \eqref{eq:sine_mimic_implication2} is
$$
\hat{g}(0).
$$
These expressions are not equal if $\eta$ is chosen such that $\hat{\eta}(\xi)\geq 0$ for all $\xi$ and $\hat{\eta}$ is supported in a sufficiently small neighborhood of 
the point $(1/a-1, -(1/a-1))$ with $\hat{\eta}(1/a-1, -(1/a-1)) \neq 0$, since in this case
\begin{equation}
\label{eq:sine_mimic_implication3}
\sum_{j \in (a^{-1}\mathbb{Z})^2} \hat{g}(j) = \hat{g}(0) - \int_{\mathbb{R}} \hat{\eta}(1/a - \xi, - 1/a +\xi)(1-|\xi|)_+ \, d\xi
\end{equation}
due to \eqref{eq:2point_ghat_eval} and the facts that $\hat{\eta}(j) = 0$ for any $j \in (a^{-1}\mathbb{Z})^2$ and $\hat{\eta}(j_1 - \xi, j_2 +\xi) = 0$ for all $\xi \in (-1,1)$ if 
$j \in (a^{-1}\mathbb{Z})^2$ unless $j = (1/a, -1/a)$ (or possibly $j = 0$ if $a = 1/2$). But then \eqref{eq:sine_mimic_implication3} is not equal to $\hat{g}(0)$ since $\hat{\eta}(1/a-1, -(1/a-1)) \neq 0$. 

This shows that \eqref{eq:sine_mimic_implication1} cannot equal \eqref{eq:sine_mimic_implication2}, a contradiction.
\end{proof}

\subsection{No mimicry for $a > \frac{1}{2}$ and $B \geq\frac{1}{2a}$}
\label{subsec:sinekernel_nonexistence2}
Finally we prove part (iii) of Theorem \ref{thm:fullrange_sinekernel}. This proof is rather more involved than the other proofs in this paper, and we break it into three steps: 

\begin{enumerate}
\item
in \textbf{step 1}, we show that band-limited mimicry can be extended to a slightly more general class of test-functions $\eta$ than Schwartz-class;
\item 
in \textbf{step 2} we develop some computations for the sine-determinant involving  a particular set of functions $h_{a, \ell}(x)$  allowed by step 1, which vanish on $a\mathbb{Z}$ 
except at $x=0$ or $x=\ell$, where $\ell$ is an odd multiple of $a$.
\item
in \textbf{step 3} we suppose the sine process can be mimicked for the relevant $a$ and $B$ and obtain a contradiction through a violation of suitable moment inequalities, 
as the parameter $\ell \to \infty$.
\end{enumerate} 

\textit{\textbf{Step 1:}} We extend the class of test functions outside the Schwartz class, to which band-limited mimicry can be applied:

\begin{lem}
\label{lem:bandlimited_bootstrapping}
If $u$ and $v$ are  u.l.m. point processes and $u$ mimics $v$ at bandwidth $[-B, B]$, then for all $n\geq 1$ if $\eta \in C(\mathbb{R}^n)$ is a function that can be written as 
$$\eta(x_1,...,x_n) = h(x_1)\cdots h(x_n)$$
with
\begin{enumerate}
\item $\hat{h}(\xi) = \int_{-\infty}^\xi \sigma(t)\,dt$ where $\sigma$ is of bounded variation, and
\item $\sigma$ and $\hat{h}$ are supported in $[-B,B]$
\end{enumerate} 
then we have
$$
\mathbb{E} \distinctsum \eta(u_{j_1},...,u_{j_n}) = \mathbb{E} \distinctsum \eta(v_{j_1},...,v_{j_n}).
$$
\end{lem}

The proof of Lemma \ref{lem:bandlimited_bootstrapping} is given in Appendix  \ref{subsec:bandlimited_bootstrapping}. The proof is a refinement of the proof of Theorem \ref{thm:correlations_to_rapidlydecaying} in Appendix \ref{sec:correlation_results}.

The point of Lemma \ref{lem:bandlimited_bootstrapping} is that $\eta$ is just slightly out of the Schwartz class, but expectations of these statistics can still be taken.

\vspace{10pt}

\textit{\textbf{Step 2:}}
We fix $a > 0$ and let $\ell=(2k+1)a$ be an odd multiple of $a$. Define the functions
$$
h_{a,\ell}(x) = S\Big(\frac{x}{a}\Big) + S\Big(\frac{x-\ell}{a}\Big),
$$
As $\ell$ is an odd multiple of $a$, we have 
$$h_{a, \ell}(x) = \frac{\sin \frac{\pi x}{a}}{\frac{\pi x}{a}}- \frac{\sin \frac{\pi x}{a}}{\frac{\pi x}{a}-(2k+1)} = O(\frac{1}{1+x^2})
$$
as $|x| \to \infty$. These functions $h_{a, \ell}$ are included among the test functions $h$ allowed in Lemma \ref{lem:bandlimited_bootstrapping}. A key property of the function $h_{a, \ell}(x)$ is that it vanishes at all $x \in a\mathbb{Z}$ except $x=0$ and $x=\ell$, where it takes the value $1$. We set
$$
H_{a,\ell}(x_1,...,x_n) = h_{a,\ell}(x_1)\cdots h_{a,\ell}(x_n),
$$
and note that  $H_{a,\ell}(x_1,...,x_n) = O(\tfrac{1}{1+x_1^2}\cdots\tfrac{1}{1+x_n^2})$, with implicit constants depending on $a, \ell, n$. Furthermore we define
$$
\Phi_n(a) = \lim_{\ell\rightarrow\infty,\, \textrm{odd}} \int_{\mathbb{R}^n} H_{a,\ell}(x_1,...,x_n) \det_{n\times n}[S(x_i-x_j)]\, d^n x.
$$
(The limit is taken over odd multiples of $a$, as $\ell\rightarrow\infty$.) Because of the decay of $H_{a,\ell}$ the integrals inside the limit are well-defined, though it is not yet obvious that the limit  exists.

\begin{lem}
\label{lem:phi_computations}
The limit defining $\Phi_n(a)$ exists for all $n\geq 1$ and $a > 0$, and
\begin{align*}
\Phi_1(a) &= 2a \\
\Phi_2(a) &= 
\begin{cases} 
2a^2,  & \mathrm{if}\; a \in (0,1/2] \\
1/2 - 2a + 4a^2, & \mathrm{if}\; a \in (1/2,\infty),
\end{cases} \\
\Phi_3(a) &=
\begin{cases}
0 & \mathrm{if}\; a \in (0,1/2] \\
(2a-1)^3 & \mathrm{if}\; a \in (1/2,\infty)
\end{cases}\\
\Phi_4(a) &=
\begin{cases}
0 & \mathrm{if}\; a \in (0,1/2] \\
(a-1/2)^2(1-20a + 12a^2) & \mathrm{if}\; a \in (1/2,1] \\
17/4 - 22a + 48a^2 - 48a^3 + 16a^4 & \mathrm{if}\; a \in (1,\infty).
\end{cases}
\end{align*}
\end{lem}

\begin{proof}
In the first place, note
\begin{equation}
\label{eq:h_fourier}
\hat{h}_{a,\ell}(\xi) = a\cdot (1 + e(-\ell \xi)) I_a(\xi),
\end{equation}
where for notational reasons we write $I_a(\xi) = \mathbf{1}_{[-1/2a,1/2a]}(\xi)$. Fix $n$ and $a$, and for $x \in \mathbb{R}^n$, let
$$
g_\ell(x) = H_{a,\ell}(x) \det_{n\times n}[S(x_i-x_j)].
$$
Using \eqref{eq:ghat_eval}, and recalling the notational convention $E = [-1/2,1/2]$, we see
\begin{multline}
\label{eq:fourier_determinant}
\int_{\mathbb{R}^n} H_{a,\ell}(x) \det_{n\times n}[S(x_i-x_j)]\, d^nx \\ 
= \hat{g}_\ell(0)
= \sum_{\sigma \in \mathfrak{S_n}} \sgn(\sigma) \int_{E^n} a^n \prod_{j=1}^n \big(1+e(\ell(\xi_j-\xi_{\sigma^{-1}(j)}))\big) I_a(\xi_j - \xi_{\sigma^{-1}(j)})\, d^n\xi.
\end{multline}
We will take the limit of this expression as $\ell\rightarrow\infty$. By multiplying cross terms of \eqref{eq:fourier_determinant}, using the Riemann-Lebesgue Lemma to eliminate any terms in which an exponential remains, we see the limit as $\ell \rightarrow \infty$ exists and
\begin{equation}
\label{eq:phi_first_eval}
\Phi_n(a) = \sum_{\sigma \in \mathfrak{S_n}} \sgn(\sigma) N(\sigma) a^n \int_{E^n} \prod_{j=1}^n I_a(\xi_{j} - \xi_{\sigma^{-1}(j)})\, d^n\xi,
\end{equation}
where
\begin{align*}
N(\sigma) &= \#\{ T \subseteq \{1,...,n\}:\; \sigma(T) = T\} \\
&= 2^{\omega(\sigma)},
\end{align*}
with $\omega(\sigma)$ the number of cycles in the permutation $\sigma$. To deduce the remainder of the Lemma one evaluates the integrals on the right side of \eqref{eq:phi_first_eval} noting that the integral in \eqref{eq:phi_first_eval} breaks into separate parts for each cycle of $\sigma$.

To evaluate the integrals, for $\nu \geq 2$ we define
$$
f_\nu(r) = \int_{E^n} \mathbf{1}_{[-r,r]}(\xi_1-\xi_2) \cdots \mathbf{1}_{[-r,r]}(\xi_{n-1}-\xi_n)\mathbf{1}_{[-r,r]}(\xi_n-\xi_1)\, d^n \xi.
$$
One can verify
\begin{align*}
f_2(r) &= 
\begin{cases} 
2r - r^2,  & \mathrm{if}\; r \in (0,1) \\
1, & \mathrm{if}\; r \in [1,\infty),
\end{cases} \\
f_3(r) &=
\begin{cases}
3r^2-2r^3 & \mathrm{if}\; r \in (0,1) \\
1 & \mathrm{if}\; r \in [1,\infty)
\end{cases}\\
f_4(r) &=
\begin{cases}
(16r^3-14r^4)/3 & \mathrm{if}\; r \in (0,1/2) \\
(1-8r+24r^2-16r^3+2r^4)/3 & \mathrm{if}\; r \in [1/2,1) \\
1 & \mathrm{if}\; r \in [1,\infty).
\end{cases}
\end{align*}
(A computer algebra system is helpful here.) Painstakingly inserting these into \eqref{eq:phi_first_eval} yields the computations of $\Phi_1,...,\Phi_4$ that have been claimed.
\end{proof}

\begin{remark}
Using  cycle index polynomials one can make the computation indicated in the last line of the above proof more efficient by noting that if $Z(\mathfrak{S_n};\, a_1,...,a_n)$ is the cycle index polynomial of $\mathfrak{S_n}$ in the variables $a_1,...,a_n$, the formula \eqref{eq:phi_first_eval} simplifies to
$$
\Phi_n(a) = (-1)^n n! a^n Z(\mathfrak{S_n}; -2 f_1(1/2a),...,-2f_n(1/2a)),
$$
where we adopt the convention $f_1(r) = 1$ for all $r$.
\end{remark}

\textit{\textbf{Step 3:}} We can now complete the last part of the proof of Theorem \ref{thm:fullrange_sinekernel}.

\begin{proof}[Proof of Theorem \ref{thm:fullrange_sinekernel}, part (iii)]
Take $a > 1/2$. We now suppose that the sine  process $z$ can be mimicked at a bandwidth $B \geq 1/2a$ by a u.l.m. point process $z'$ supported on $a\mathbb{Z}$, and we will obtain a contradiction. For $\ell$ always an odd multiple of $a$, consider the random variable
\begin{align}
\label{eq:lucky_indicator}
X_\ell &= \sum_{j} h_{a,\ell}(z'_j) \\
&= \#_{\{0,\ell\}}(z'),
\end{align}
with the second identity dependent on the assumption that $z'$ is supported on $a\mathbb{Z}$.

We consider two sets of inequalities satisfied by expectation values of functions of this random variable. First,
$X_\ell$ is an \emph{nonnegative} \emph{integer-valued} random variable, and so clearly
\begin{equation}
\label{eq:X_ell_factorialmoment}
\mathbb{E}\, X_\ell (X_\ell-1) (X_\ell-2) (X_\ell-3) \geq 0.
\end{equation}

Secondly, let  
$$
m_\ell^r := \mathbb{E}\, X_\ell^r
$$
denote the {\em $r$-th moment} of $X_\ell$. By a consequence of the Hamburger moment criterion (see e.g. \cite[Theorem 1.2]{ShTa43}), we have,
\begin{equation}
\label{eq:Hamburger_det}
D_\ell = \det\begin{pmatrix} 
m_\ell^0 & m_\ell^1 & m_\ell^2 \\
m_\ell^1 & m_\ell^2 & m_\ell^3 \\
m_\ell^2 & m_\ell^3 & m_\ell^4
\end{pmatrix} \geq 0,
\end{equation}

We claim that for any choice of $a > \frac{1}{2}$ at least one of the inequalities \eqref{eq:X_ell_factorialmoment} or \eqref{eq:Hamburger_det} will not hold for all sufficiently large $\ell$.

For consider first $a \in (1/2,1]$. Note that from \eqref{eq:lucky_indicator} \and the indicator function identity \eqref{eq:factorial_counts} we have
\begin{equation}
\label{eq:counts_to_H}
\mathbb{E}\, X_\ell (X_\ell-1) (X_\ell -2) (X_\ell-3) = \mathbb{E} \distinctsum H_{a,\ell}(z'_{j_1}, z'_{j_2}, z'_{j_3}, z'_{j_4}).
\end{equation}
The computation \eqref{eq:h_fourier} reveals that $\hat{h}_{a,\ell}(\xi) = -2\pi i a\ell \int_{-\infty}^\xi  \, e(-\ell t) I_a(t)\, dt$, with the integrand of bounded variation and supported in $[-1/2a,1/2a] \subset [-B,B]$, so Lemma \ref{lem:bandlimited_bootstrapping} may be applied; if $z'$ mimics $z$, then \eqref{eq:counts_to_H} is equal to
$$
\int_{\mathbb{R}^4} H_{a,\ell}(x) \det_{4\times 4}[S(x_i-x_j)]\, d^4x.
$$
Taking the limit of this expression as $\ell \rightarrow \infty$ along odd multiples of $a$, Lemma \ref{lem:phi_computations} yields
\begin{align*}
\lim_{\ell \rightarrow\infty, \mathrm{odd}} \mathbb{E}\,X_\ell (X_\ell -1) (X_\ell-2) (X_\ell-3) &= \Phi_4(a) \\
&= (a-1/2)^2 (1-20a + 12a^2).
\end{align*}
For $a \in (1/2,1]$, it can be checked that this number is strictly negative, but this contradicts \eqref{eq:X_ell_factorialmoment}.

Now consider $a > 1$. As above we have
$$
\lim_{\ell\rightarrow\infty, \mathrm{odd}} \mathbb{E}\, X_\ell (X_\ell-1) \cdots (X_\ell - (n-1)) = \Phi_n(a),
$$
and from this, using Lemma \ref{lem:phi_computations}, one may extract
\begin{align*}
\lim_{\ell\rightarrow\infty, \mathrm{odd}} \mathbb{E}\, X_\ell & = 2a,\quad (\textrm{for}\; a > 0) \\
\lim_{\ell\rightarrow\infty, \mathrm{odd}} \mathbb{E}\, X_\ell^2 &= \frac{1}{2} + 4a^2,\quad  (\textrm{for}\; a > 1/2)\\
\lim_{\ell\rightarrow\infty, \mathrm{odd}} \mathbb{E}\, X_\ell^3 &= \frac{1}{2} + 2a + 8a^3, \quad (\textrm{for}\; a > 1/2)\\
\lim_{\ell\rightarrow\infty, \mathrm{odd}} \mathbb{E}\, X_\ell^4 &= \frac{7}{4} + 2a + 4a^2 + 16a^4, \quad (\textrm{for}\; a > 1),
\end{align*}
and further, using the notation in \eqref{eq:Hamburger_det}, one may compute
$$
\lim_{\ell\rightarrow\infty, \mathrm{odd}} D_\ell = \frac{1}{2} - a^2, \quad (\textrm{for}\; a > 1).
$$
(A computer algebra system is helpful here.) But this is strictly negative for any choice of $a \in (1,\infty)$, and this contradicts \eqref{eq:Hamburger_det}.

Thus we have obtained a contradiction for all $a > 1/2$, so in this range such a u.l.m. point process $z'$ does not exist.
\end{proof}

\section{Further questions}
\label{sec:conclusion}

This paper formulated  the band-limited mimicking problem for u.l.m. point processes on $\mathbb{R}$. We studied two such processes where band-limited mimicry is possible, the Poisson process and the sine process. These processes are  special in at least two ways:
\begin{enumerate}
\item
Both  processes  are  $\mathbb{R}$ translation-invariant, in probability law and in the correlation sense defined in Section \ref{subsec:mainresults}. 

\item
These processes have $n$-point correlation measures for each $n \ge 1$ that have absolutely continuous densities $d \rho_n(x_1, x_2, ..., x_n)= f_n(x_1, ..., x_n) dx_1 dx_2 ... dx_n$, with $f_n(x_1,x_2, \cdots, x_n)$ defined on $\mathbb{R}^n$, with the property that they holomorphically extend to entire functions $f_n(z_1, z_2, ..., z_n)$ on $\mathbb{C}^n$.
\end{enumerate}
 
We raise several  general questions.

First, we do not know to what extent the band-limited mimicry phenomenon discussed in this paper  exists for   general u.l.m. point processes. Are there  u.l.m. point processes $\mathcal{P}$  that do not permit  band-limited mimicry at any bandwidth $B>0$? If there are, how general is the class of such processes for which band-limited mimicry  exists for some $(a,B) $ with $B>0$?

Second, related to this question: which u.l.m. point processes $u$ have the property that if $u$ supports band-limited mimicry for some $B>0$ on a lattice $a\mathbb{Z}$ then it supports band-limited mimicry for some $B= B(a') >0$ on each lattice $a' \mathbb{Z}$ having $0 < a' \le a$? Does this class of processes $u$ include all $\mathbb{R}$-translation invariant u.l.m. point processes?

Third, what restrictions does band-limited mimicry entail for point processes not necessarily supported on a lattice? For instance, let $\mathcal{T}_1$ be the class of all u.l.m. point processes $u$ which mimic the sine process at a bandwidth $B=1$, and let
$$
\mu = \sup\; \{m:\; \textrm{there exists } u \in \mathcal{T}_1 \textrm{ such that almost surely } |u_i-u_j| \geq m \textrm{ for all } i\neq j\}.
$$
Theorem \ref{thm:fullrange_sinekernel} shows that $\mu \geq 1/2$. The method of proof in Carneiro et. al \cite{CaChLiMi17}, which makes use only of pair correlation, should be able to be straightforwardly modified to show that $\mu \leq .606894$. It may be that $\mu = 1/2$.

Likewise let
$$
\lambda := \inf\; \{\ell:\; \textrm{there exists } u \in \mathcal{T}_1 \textrm{ such that almost surely } |u_{j+1}-u_j| \leq \ell \textrm{ for all } j \in \mathbb{Z}\}.
$$
What is the value of $\lambda$? Is it finite? It may be that a reinterpretation of methods from number theory (see e.g. Soundararajan \cite{So96}) can yield further upper bounds for $\mu$ and lower bounds for $\lambda$. Questions about both $\mu$ and $\lambda$ are closely connected to classical questions about gaps between zeros of the Riemann zeta function.

Fourth, to what extent do classical theorems and conjectures about the sine process (or zeros of the zeta function or eigenvalues of a random matrix) remain true for a point process which merely mimics the sine-process at some bandwidth? For instance, central limit theorems for mesoscopic statistics will still hold for processes which only mimic the sine-process (see \cite[Sec. 7]{Ro14}), along with suitably interpreted central limit theorems for characteristic polynomials (using the method of \cite[Sec. 7]{DiEv01}). To take another example, to what extent do results and conjectures about extreme values of the zeta function or characteristic polynomials (see e.g. \cite{ArBeBo17,ArBeBoRaSo19,ChMaNa18,FyHiKe12,Na18,PaZe17}) rely only upon information preserved by band-limited mimicry?

We also raise some more specific questions.

First, Theorem \ref{thm:fullrange_sinekernel} of this paper did  not completely determine the parameter ranges of $a$ and $B$  permitting band-limited mimicry for the sine process. What happens for those $(a, B)$ in  the white region of Figure \ref{fig:GoNoGoPlots}? Can the sine process be mimicked there or not? 
 
Second, it is obviously of interest to investigate the extent to which the band-limited mimicry phenomenon extends to other point processes.  Two  one-parameter classes of point processes which may be of interest to study  are: 
 
\begin{enumerate}
\item 
Valk\'{o} and Vir\'{a}g \cite{ValVir09}  define the one-parameter  family of ${\rm Sine}_{\beta}$ processes,  where $\beta> 0$. All members of this one-parameter family  are  $\mathbb{R}$-translation invariant, and they have the Poisson process as a suitable scaling limit as $\beta \to 0$, see Allez and Dumaz \cite{AlDu14}. The sine-process corresponds to $\beta=2$, and the Gaussian orthogonal and symplectic ensembles corresponds to  $\beta=1$ and $4$ respectively.
 
\item 
Sodin \cite{Sod17} introduces the one-parameter family of $\mathfrak{Si}_{a}$-processes (for $a \in \mathbb{R}$) as a model of critical points of characteristic polynomials for random matrices. In particular, the $\mathfrak{Si}_{0}$-process is presented as a model for the limiting  distribution of (normalized) spacings of zeros of the derivative of the Riemann $\xi$-function $\xi(s) = s(s-1) \pi^{-\frac{s}{s}} \Gamma (\frac{s}{2}) \zeta(s)$ assuming RH and the multiple correlation conjecture,  cf. \cite[Corollary 2.3]{Sod17}. (The multiple correlation conjecture is equivalent to the GUE Hypothesis, in the form \cite[Conjecture 2.1]{LaRo19}.)
 \end{enumerate}

%

\vspace{5pt}

\noindent {\bf Acknowledgments.}
We thank the reviewers for helpful comments. Work of the  first author was partially supported by NSF grant DMS-1701576, a Chern Professorship at MSRI in Fall 2018 and by a  Simons Foundation Fellowship in 2019. MSRI is partially supported by an NSF grant. The second author was partially supported by NSF grant DMS-1701577 and by an NSERC grant.

\appendix
\section{ Some general results on correlation measures}
\label{sec:correlation_results}

In this appendix we collect and prove some results regarding the correlation functions of point processes which we have used in the paper.

\subsection{Existence of correlation measures}
\label{subsec:existence_corrmeasures}

The following result essentially is \cite[Prop 3.2]{Le75}. We include the simple proof here for completeness.

\begin{thm}
\label{thm:moments_implies_correlation}
If $u$ is a point process on $\mathbb{R}$ such that for any compact set $K$ the random variable $\#_K(u)$ has finite moments of all orders, then for all $n\geq 1$ there exists a unique Borel measure $\rho_n$ on $\mathbb{R}^n$ such that
\begin{equation}
\label{eq:correlations}
\mathbb{E} \distinctsum \phi(u_{j_1},...,u_{j_n}) = \int_{\mathbb{R}^n} \phi(x_1,...,x_n) d\rho_n(x_1,...,x_n),
\end{equation}
for all $\phi \in C_c(\mathbb{R}^n)$.
\end{thm}

\begin{remark}
A point process with uniform local moments will satisfy the hypothesis of  Theorem \ref{thm:moments_implies_correlation}.
\end{remark}

\begin{proof}
The fact that $\#_K(u)$ has finite $n$-th moment for any compact $K$ implies that for $\phi \in C_c(\mathbb{R}^n)$, the random variables $\distinctsum \phi(u_{j_1},...,u_{j_n})$ are integrable, and thus the mapping $\Lambda$ defined by
$$
\Lambda \phi = \mathbb{E} \distinctsum \phi(u_{j_1},...,u_{j_n}),
$$
is a positive linear functional on $C_c(\mathbb{R}^n)$. The Riesz representation theorem \cite[Ch. 2, Theorem 2.14]{Ru87} thus implies the existence of the Borel measure $\rho_n$.
\end{proof}

\subsection{Bootstrapping test functions from $C_c(\mathbb{R}^n)$ to $\mathcal{S}(R^n)$.}
\label{subsec:compactlysupported_to_schwartz}

We show that for u.l.m. point processes the correlation measures make sense with respect to not only $C_c(\mathbb{R}^n)$ test functions, but also Schwartz class test functions. Actually we show a bit more.

\begin{thm}
\label{thm:correlations_to_rapidlydecaying}
If $u$ is a u.l.m.  point process on $\mathbb{R}$ and $\rho_n$ is the measure on $\mathbb{R}^n$ defined by \eqref{eq:correlations} for $\phi \in C_c(\mathbb{R}^n)$, then \eqref{eq:correlations} also holds for all $n\geq 1$ and all $\phi \in C(\mathbb{R}^n)$ such that
$$
\phi(x_1,...,x_n) = O\Big(\frac{1}{(1+x_1^2)\cdots (1+x_n^2)}\Big).
$$
\end{thm}

\begin{remark} 
Hence in particular for a point process with uniform local moments, \eqref{eq:correlations} holds for all $\phi \in \mathcal{S}(\mathbb{R}^n)$, for all $n\geq 1$.
\end{remark}

\begin{proof}
Let
\begin{equation}
\label{eq:Q_def}
Q(x_1,...,x_n) = \frac{1}{1+x_1^2} \cdots \frac{1}{1+x_n^2}.
\end{equation}
We first establish for the point process $u$ that
\begin{equation}
\label{eq:quad_convergence}
\mathbb{E} \distinctsum Q(u_{j_1},...u_{j_n}) < +\infty.
\end{equation}
For, there exists a absolute constant $K$ such that
$$
Q(x_1,...,x_n) \leq K \sum_{L \in \mathbb{Z}^n} Q(L_1,...,L_n) \mathbf{1}_{[L_1,L_1+1]}(x)\cdots \mathbf{1}_{[L_n,L_n+1]}(x)
$$
for all $x\in \mathbb{R}$, so we have that
\begin{align*}
\mathbb{E} \distinctsum Q(u_{j_1},...,u_{j_n}) &\leq K\cdot \, \mathbb{E} \sum_{L \in \mathbb{Z}^n} Q(L_1, ..., L_n) \#_{[L_1,L_1+1]}(u)\cdots \#_{[L_n,L_n+1]}(u) \\
& \leq K \sum_{L \in \mathbb{Z}^n}  Q(L_1,...,L_n) \prod_{i=1}^n (\mathbb{E}\, \#_{[L_i,L_i+1]}(u)^n)^{1/n} \\
& \leq K C_n \sum_{L \in \mathbb{Z}^n} Q(L_1,...,L_n) < +\infty,
\end{align*}
using Fatou's lemma and H\"older's inequality in the second line.

For the same reasons, we have
\begin{equation}
\label{eq:quad_integrable}
\int_{\mathbb{R}^n} Q(x_1,..,x_n) \, d\rho_n(x_1,...,x_n) < +\infty.
\end{equation}

Note also that \eqref{eq:quad_convergence} implies that almost surely
\begin{equation}
\label{eq:quad_summable}
\distinctsum Q(u_{j_1},...,u_{j_n}) \quad \mathrm{converges}.
\end{equation}

Let $\beta \in C_c(\mathbb{R}^n)$ be a bump function takes the value $1$ in some neighborhood of $0\in \mathbb{R}^n$ and which satisfies $0\leq \beta(x) \leq 1$ for all $x\in \mathbb{R}^n$. For $R > 0$ define $\phi_R(x) = \phi(x) \beta(x/R)$, and note for all $x\in \mathbb{R}^n$,
$$
\lim_{R\rightarrow\infty}\phi_R(x) = \phi(x).
$$
Moreover $\phi_R \in C_c(\mathbb{R}^n)$ for all $R$, and by assumption there is a constant $C>0$ such that 
$$
|\phi_R(x_1,...,x_n)| \leq C \,Q(x_1,...,x_n)
$$
for all $x\in \mathbb{R}^n$.

Now from \eqref{eq:quad_summable}, it is easy to see that almost surely
$$
\lim_{R\rightarrow\infty} \distinctsum \phi_R(u_{j_1},...,u_{j_n}) = \distinctsum \phi(u_{j_1},...,u_{j_n}).
$$
Hence using \eqref{eq:quad_convergence},\eqref{eq:quad_integrable} and dominated convergence,
\begin{align*}
\mathbb{E} \distinctsum \phi(u_{j_1},...,u_{j_n}) &= \lim_{R\rightarrow\infty} \distinctsum \phi_R(u_{j_1},...,u_{j_n}) \\
&= \lim_{R\rightarrow\infty} \int_{\mathbb{R}^n} \phi_R(x_1,...,x_n) \, d\rho_n(x_1,...,x_n) \\
&= \int_{\mathbb{R}^n} \phi(x_1,...,x_n)\, d\rho_n(x_1,...,x_n),
\end{align*}
as claimed.
\end{proof}

\subsection{Bootstrapping band-limited test functions}
\label{subsec:bandlimited_bootstrapping}

We prove Lemma \ref{lem:bandlimited_bootstrapping}. The proof will involve similar ideas to that of Theorem \ref{thm:correlations_to_rapidlydecaying}. We require two lemmas from analysis first. 

Below we consider functions $\sigma$ which are of bounded variation. We use the notation $\var(\sigma)$ to denote the total variation of the function $\sigma$ on the real line.

\begin{lem}
\label{lem:bv_fourier_decay}
Suppose $s(\xi) = \int_{-\infty}^\xi \sigma(t)\, dt$ where $\sigma$ and $s$ are integrable and $\sigma$ is of bounded variation. Then
$$
\hat{s}(x) = O(\min(\|s\|_{L^1(\mathbb{R})}, \var(\sigma)/x^2)).
$$
\end{lem}

\begin{proof}
This is a combination of two standard results. The bound $\hat{s}(x) \leq \|s\|_{L^1}$ is obvious, and the bound $\var(\sigma)/x^2$ comes from integrating by parts twice in computing the Fourier transform:
$$
|\hat{s}(x)| = \Big| \int_{-\infty}^\infty \frac{e^{-i2\pi x\xi}}{(-i2\pi x)^2} d\sigma(\xi)\Big| \leq \frac{1}{4\pi^2 x^2} \int_{-\infty}^\infty |d\sigma(\xi)|.
$$
Combining these bounds proves the lemma.
\end{proof}

\begin{lem}
\label{lem:sigma_approx}
If $\sigma(t)$ is supported on the interval $[A,B]$ and of bounded variation, then for any $\epsilon > 0$ there exists a Schwartz function $\tilde{\sigma}(t)$ supported on $[A,B]$ such that
$$
\var(\tilde{\sigma}) \leq \var(\sigma),
$$
and
$$
\|\tilde{\sigma}-\sigma \|_{L^1(\mathbb{R})} < \epsilon.
$$
\end{lem}

\begin{proof}
As $\sigma$ is of bounded variation, the Jordan decomposition (see \cite[Sec. 5.2]{Ro88}) tells us there exists monotonic nondecreasing functions $\sigma_+$ and $\sigma_-$ such that $\sigma = \sigma_+-\sigma_-$ and $\sigma_+$ and $\sigma_-$ are constant for $t\notin [A,B]$, that is
$$
\sigma_\pm(t) = \sigma_\pm(A), \quad \textrm{for all}\; t \leq A,
$$
$$
\sigma_\pm(t) = \sigma_\pm(B), \quad \textrm{for all}\; t\geq B,
$$
and moreover $\var(\sigma) = \var(\sigma_+)+\var(\sigma_-)$. It is a straightforward exercise to construct monotonically nondecreasing functions $\tilde{\sigma}_+$ and $\tilde{\sigma}_-$ with Schwartz class derivatives such that for either $+$ or $-$,
$$
\|\tilde{\sigma}_\pm - \sigma_\pm\|_{L^1} < \epsilon /2,
$$
$$
\tilde{\sigma}_\pm(t) = \sigma_\pm(A), \quad \textrm{for all}\; t\leq A,
$$
$$
\tilde{\sigma}_\pm(t) = \sigma_\pm(B), \quad \textrm{for all}\; t \geq B.
$$
Note $\var(\tilde{\sigma}_+) = \var(\sigma_+) = \sigma_+(B) - \sigma_-(A)$ and $\var(\tilde{\sigma}_-) = \var(\sigma_-) = \sigma_-(B)-\sigma_-(A)$, so if $\tilde{\sigma} = \tilde{\sigma}_+ - \tilde{\sigma}_-$,
$$
\var(\tilde{\sigma}) \leq \var(\tilde{\sigma}_+) + \var(\tilde{\sigma}_-) = \var(\sigma_+) + \var(\sigma_-) = \var(\sigma),
$$
verifying the first inequality of the lemma. Because $\tilde{\sigma}$ is compactly supported and is the difference of two functions with Schwartz class derivatives, $\tilde{\sigma}$ is itself Schwartz class. Moreover from the triangle inequality, 
$$
\|\tilde{\sigma}-\sigma\|_{L^1} \leq \|\tilde{\sigma}_+-\sigma_+\|_{L^1} + \|\tilde{\sigma}_--\sigma_-\|_{L^1} < \epsilon,
$$
verifying the second claim of the lemma.
\end{proof}

We finally turn to Lemma \ref{lem:bandlimited_bootstrapping}.

\begin{proof}[Proof of Lemma \ref{lem:bandlimited_bootstrapping}]
The proof follows that of Theorem \ref{thm:correlations_to_rapidlydecaying}. We show for all $R \geq 1$ there exists $\eta_R \in \mathcal{S}(\mathbb{R}^n)$ such that
\begin{equation}
\label{eq:eta_property1}
\lim_{R\rightarrow\infty} \eta_R(x) = \eta(x), \quad \textrm{for all}\; x\in \mathbb{R}^n,
\end{equation}
\begin{equation}
\label{eq:eta_property2}
\supp \hat{\eta}_R \subset [-B,B]^n\, \quad \textrm{for all}\; R \geq 1,
\end{equation}
\begin{equation}
\label{eq:eta_property3}
\eta_R(x) = O(Q(x)), \quad \textrm{for all}\; x\in \mathbb{R}^n,\, R \geq 1,
\end{equation}
where $Q$ is the the quadratically decaying function defined in \eqref{eq:Q_def}. Then exactly by the argument in the proof of Theorem \ref{thm:correlations_to_rapidlydecaying}, we have
\begin{align*}
\mathbb{E} \distinctsum \eta(u_{j_1},...,u_{j_n}) &= \lim_{R\rightarrow\infty} \mathbb{E} \distinctsum\eta_R(u_{j_1},...,u_{j_n}) \\
&= \lim_{R\rightarrow\infty} \mathbb{E} \distinctsum \eta_R(v_{j_1},...,v_{j_n}) \\
&= \mathbb{E} \distinctsum \eta(v_{j_1},...,v_{j_n}).
\end{align*}
The functions $\eta_R$ are constructed in the following way. For $\sigma$ as in statement of Lemma \ref{lem:bandlimited_bootstrapping}, let $\tilde{\sigma}_R$ be a function described by Lemma \ref{lem:sigma_approx} such that $\supp \tilde{\sigma}_R \subset [-B,B]$, $\var(\tilde{\sigma}_R) \leq \var(\sigma)$ and $\|\tilde{\sigma}_R - \sigma\|_{L^1} \leq 1/R$. Define $h_R$ by
$$
\hat{h}_R(\xi) = \int_{-\infty}^\xi \tilde{\sigma}_R(t)\,dt,
$$
and note that
\begin{equation}
\label{eq:h_property2}
\supp \hat{h}_R \subset [-B,B],
\end{equation}
and for all $\xi$, $|\hat{h}_R(\xi) - \hat{h}(\xi)| \leq 1/R$ so that from the support of both functions $\hat{h}, \hat{h}_R$,
\begin{equation}
\label{eq:h_property1}
|h_R(x) - h(x)| \leq \| \hat{h}_R - \hat{h} \|_{L^1} \leq 2B/R.
\end{equation}
Finally from Lemma \ref{lem:bv_fourier_decay}, we have
\begin{align}
\label{eq:h_property3}
\notag h_R(x) &= O(\min(\|\hat{h}_R\|_{L^1}, \var(\tilde{\sigma}_R)/x^2)) \\
\notag &= O(\min(2B \|\tilde{\sigma}_R\|_{L^1}, \var(\sigma)/x^2)) \\
&= O\Big(\frac{1}{1+x^2}\Big).
\end{align}
Letting $\eta_R(x_1,...,x_n) = h_1(x_1)\cdots h_n(x_n)$, we see that \eqref{eq:eta_property1}, \eqref{eq:eta_property2}, \eqref{eq:eta_property3} are satisfied, using \eqref{eq:h_property1}, \eqref{eq:h_property2}, \eqref{eq:h_property3} respectively. This completes the proof.
\end{proof}

\end{document}